\documentclass[12pt,letterpaper]{amsart}
\usepackage{graphicx}
\usepackage{amsmath}
\usepackage{amssymb}
\usepackage{amsfonts}
\usepackage{amsthm}
\usepackage{cancel}
\usepackage[all]{xy}
\usepackage[neveradjust]{paralist}
\usepackage{enumitem}
\usepackage{multicol}
\usepackage{mathrsfs}
\usepackage{mathdots}
\usepackage[pagebackref,colorlinks=true,linkcolor=blue,citecolor=blue]{hyperref}

\usepackage{color}

\newcommand{\N}{\mathbb{N}}

\newcommand{\C}{\mathbb{C}}

\newcommand{\B}{{\mathcal B}}

\newcommand{\inv}{^{-1}}

\newcommand{\GL}{\mathrm{GL}}
\newcommand{\SO}{\mathrm{SO}}

\newcommand{\comment}[1]{}

\newtheorem{thm}{Theorem}[section]
\newtheorem{cor}[thm]{Corollary}
\newtheorem{lemma}[thm]{Lemma}
\newtheorem{prop}[thm]{Proposition}
\newtheorem {conj}[thm]{Conjecture}

\newtheorem {ques/conj}[thm]{Question/Conjecture}

\numberwithin{equation}{section}

\begin{document}

\title[]{A converse theorem for quasi-split even special orthogonal groups over finite fields}

\author{Alexander Hazeltine}
\address{Department of Mathematics\\
University of Michigan\\
Ann Arbor, MI, 48109, USA}
\email{ahazelti@umich.edu}


\dedicatory{}


\begin{abstract}
    We prove a converse theorem for the case of quasi-split non-split even special orthogonal groups over finite fields. There are two main difficulties which arise from the outer automorphism and non-split part of the torus. The outer automorphism is handled similarly to the split case, while new ideas are developed to overcome the non-split part of the torus.
\end{abstract}

\maketitle

\tableofcontents

\section{Introduction}

Let $\mathrm{G}$ be a connected reductive group, $F$ be a local or finite field, and $G:=\mathrm{G}(F).$ Converse theorems aim to classify certain representations of $G$ in terms of their invariants. One such invariant is the $\gamma$-factor. These invariants are arithmetic in nature and are a vital part of the Langlands program. 

For the moment, let $F$ be a non-Archimedean local field of characteristic 0, $\pi$ be an irreducible generic representation of $\GL_l(F)$, $\tau$ be an irreducible generic representation of $\GL_n(F)$, and $\psi$ be an additive non-trivial character of $F.$ Jacquet, Piatetskii--Shapiro, and Shalika defined the local twisted Rankin-Selberg $\gamma$-factors, denoted $\gamma(s, \pi \times \tau, \psi)$ where $s\in\mathbb{C}$, using  Rankin--Selberg convolutions \cite{JPSS83} (equivalently, these can be defined by the Langlands--Shahidi method \cite{Sha84, Sha90}). Jacquet conjectured the following local converse theorem for $\GL_l(F)$.

\begin{conj}\label{lcp}
Let $\pi_1$ and $\pi_2$ be irreducible generic representations
of $\GL_l(F)$ with the same central character.
If
\[
\gamma(s, \pi_1 \times \tau, \psi) = \gamma(s, \pi_2 \times \tau, \psi),
\]
as functions of the complex variable $s$, for all irreducible
generic representations $\tau$ of $\GL_n(F)$ with $1 \leq n \leq 
[\frac{l}{2}]$, then $\pi_1 \cong \pi_2$.
\end{conj}

Conjecture \ref{lcp} was proved independently by Chai (\cite{Cha19}) and Jacquet and Liu (\cite{JL18a}) using different analytic methods. 

For general quasi-split reductive groups over non-Archimedean local fields of characteristic 0, Jiang gave a general conjecture for local converse theorems for generic representations in \cite[Conjecture 3.7]{Jia06} and many cases have since been established. A non-exhaustive list of results includes  $\SO_{2l+1}$ (\cite{JS03, Jo22}), $\mathrm{Sp}_{2l}$ (\cite{Jo22, Zha18}), $\mathrm{U}(l,l)$ (\cite{Mor18, Zha18}), $\mathrm{U}_{2l+1}$ (\cite{Zha19}), split $\SO_{2l}$ (\cite{HL22b}), and $\widetilde{\mathrm{Sp}}_{2l}$ (\cite{Haa22}). For quasi-split non-split $\SO_{2l}$, the local converse theorem has also been established by the author (\cite{Haz23b}) by generalizing the results in this paper to the local case. The arguments presented here in the finite fields case serve as a guide to the local case. However, we remark that there are significant differences between the finite and local cases (see \cite[\S1]{Haz23b} for a discussion of these difficulties).

The case for split and quasi-split $\SO_{2l}$ was also proved in \cite{HKK23} using the theta correspondence (similarly to \cite{Haa22} but using more refined arguments). We remark that there are significant differences in the approaches of \cite{Haz23b, HL22b} and \cite{HKK23}. In particular, similar to the case studied in this paper, the work of \cite{Haz23b,HL22b} utilizes the properties of partial Bessel functions to compute the zeta integrals directly. Also, by applying Langlands functoriality and Conjecture \ref{lcp}, several of the above cases would be implied by the works of Arthur and Mok (\cite{Art13, Mok15}). However, it is desirable to have direct proofs which reflect the intrinsic properties of the groups as is carried out in \cite{Haz23b, HL22b, Jo22, Zha18, Zha19}.

Over finite fields of odd characteristic, the converse theorem for generic cuspidal representations has an analogous form to the local case. Hereinafter, we let $F$ be a finite field of odd characteristic. For $\GL_l(F)$, Nien established the converse theorem for generic cuspidal representations (\cite{Nie14}) by using normalized Bessel functions and the twisted gamma factors defined by Roditty (\cite{Rod10}). In \cite{LZ22a}, Liu and Zhang defined the twisted gamma factors for generic cuspidal representations of $\mathrm{Sp}_{2l}(F)$, $\SO_{2l+1}(F)$, $\mathrm{U}_{2l}(F)$, and $\mathrm{U}_{2l+1}(F)$ and proved the corresponding converse theorems. Liu and Zhang also established the converse theorem for $\mathrm{G}_2(F)$ (\cite{LZ22b}). Building upon the ideas of \cite{LZ22a}, Liu and the author proved the converse theorem for split $\SO_{2l}(F)$ in \cite{HL22a}.

The remaining case for classical groups over finite fields is that of quasi-split non-split $\SO_{2l}(F).$ Hereinafter, by ``quasi-split,'' we mean quasi-split non-split as the analogous results for split case are contained in \cite{HL22b}. There are two key difficulties in the quasi-split case. The first is the existence of the outer automorphism. The second is presented by the maximal torus being non-split. In this paper, we develop new ideas in order to overcome these obstacles and prove the converse theorem. That is, we define the twisted gamma factors $\gamma(\pi\times \tau, \psi)$ (see Proposition \ref{gammafactor}) for an irreducible generic cuspidal representation $\pi$ of $\SO_{2l}(F)$ and an irreducible generic representation $\tau$ of $\GL_{n}(F)$, 
and prove the following theorem. 

\begin{thm}[The Converse Theorem for quasi-split $\SO_{2l}$, Theorem \ref{converse thm}]\label{converse thm intro}
Let $\pi$ and $\pi^\prime$ be irreducible cuspidal $\psi$-generic representations of the quasi-split group $\SO_{2l}(F)$ with the same central character. If $$\gamma(\pi\times\tau,\psi)=\gamma(\pi^\prime\times\tau,\psi),$$ 
for all irreducible generic representations $\tau$ of $\GL_n(F)$ with $n\leq l,$
then $\pi\cong\pi'$ or $\pi\cong \pi'{}^c,$
 where $c$ is the outer automorphism. 
\end{thm}

Theorem \ref{converse thm intro} and Corollaries \ref{conj n gamma} and \ref{conj l gamma} imply that twisted gamma factors  are unable to distinguish irreducible generic cuspidal representations $\pi$ and $\pi^c$ of $\SO_{2l}(F)$, which is a unique phenomenon for $\SO_{2l}$ among all the classical groups. This was also viewed in the split case in \cite{HL22a}. Over non-Archimedean fields, this is consistent with Arthur's work on the local Langlands correspondence and the local Langlands functoriality, and the work of Jiang and Soudry on local descent for even special orthogonal groups (see \cite{Art13} and \cite{JS12}). It remains an open problem to find a set of invariants which can distinguish between $\pi$ and $\pi^c$. For split $\SO_4$ over a $p$-adic field, this is done in \cite{YZ23} using twisted exterior square local $\gamma$-factors, but remains open in higher rank (see the discussion at the end of \cite[\S1]{YZ23} or \cite[Proposition 7.3]{Mat24}).

Now, we introduce the key ideas needed to prove Theorem \ref{converse thm intro}. Recall that there are two major difficulties in our case, the outer automorphism, which we will denote by $c$
 (see \S\ref{setup} for a definition), and the quasi-split torus. The outer automorphism is handled in a similar manner to the split case considered in \cite{HL22a}. Namely, instead of considering solely the normalized Bessel function $\mathcal{B}_{\pi,\psi}$ of $\pi$ (see \S\ref{section Bessel Functions}), we also involve the normalized Bessel function $\mathcal{B}_{\pi^c,\psi}$ of $\pi^c$ (specifically, we do this when we twist by $\GL_l(F)$) to establish the following theorem.
 
\begin{thm}[Theorem \ref{Bessels Equal}]\label{Bessels Equal intro}
Let $\pi$ and $\pi^\prime$ be irreducible cuspidal $\psi$-generic representations of the quasi-split group $\SO_{2l}(F)$ with the same central character. If $$\gamma(\pi\times\tau,\psi)=\gamma(\pi^\prime\times\tau,\psi),$$
for all irreducible generic representations $\tau$ of $\GL_n(F)$ with $1 \leq n\leq l,$
then we have that $$(\B_{\pi,\psi}+\B_{\pi^c,\psi})(g)=(\B_{\pi^\prime,\psi}+\B_{\pi'{}^c,\psi})(g)$$  for any $g\in\SO_{2l}(F)$.
\end{thm}

While the result is similar to the split case, the proofs have two main differences. The first one is that in the split case, the outer conjugation acts nontrivially on some of the simple roots and Weyl elements, while for the non-split case, the outer conjugation acts trivially on all the simple roots and Weyl elements. Consequently, the outer conjugation only nontrivially affects the torus and its effect is to partition the Bruhat cells corresponding to certain Weyl elements into two parts: those elements which are fixed by the outer conjugation and those which do not. More specifically, we see that the twists up to $\GL_{l-1}(F)$ determines the elements which are fixed by the outer conjugation (Theorem \ref{GL_{l-1}}) while the twists by $\GL_{l}(F)$ determines the  elements which are not fixed by the outer conjugation (Theorem \ref{GL_l}).

The second main difference is presented by the non-split part of the maximal torus. Indeed, in the split case, normalized Bessel functions $\mathcal{B}_{\pi,\psi}$ are nonzero on the part of the torus which is precisely the center (\cite[Lemma 4.3]{HL22a}). The proof of this relies on the fact that, for split groups, the center is precisely the intersection of all of the kernels of the simple roots. This fact is not true for quasi-split groups. In the quasi-split case, we still have that normalized Bessel functions $\mathcal{B}_{\pi,\psi}$ are nonzero on the part of the torus which is precisely the center (Lemma \ref{center}); however, the proof requires a more careful analysis of the effect of conjugating certain unipotent elements by elements of the torus. This careful analysis must also be carried out in other arguments where properties of the Bessel functions are examined in conjunction with elements of root subgroups. Note that, unlike split groups, the root subgroups of quasi-split groups are not necessarily 1-dimensional or even commutative in general.

Here is the structure of this paper. In \S\ref{setup}, the relevant groups, representations, and notions are introduced. In \S\ref{section Multiplicity one theorems and the gamma factor}, we establish multiplicity one results (Propositions \ref{MultOne} and \ref{Mult1}) and define the zeta integrals and $\gamma$-factors. In \S\ref{section Bessel Functions}, the Bessel functions are defined and relevant properties are collected. In \S\ref{section Twists by GLn}, \S\ref{section Twists by GL l-1}, and \S\ref{section Twists by GLl} we study the $\GL_n$ twists, $1 \leq n \leq l$, and determine the relation between the $\GL_n$ twists and the Bessel functions (Theorems \ref{GL_n}, \ref{GL_{l-1}}, \ref{GL_l}). In \S\ref{section The converse theorem}, we prove Theorem \ref{Bessels Equal intro} and the converse theorem (Theorem \ref{converse thm intro}).

\subsection*{Acknowledgements}

The author would like to thank Professors Baiying Liu, Dihua Jiang,and Freydoon Shahidi for their interest and constant support.
The author would also like to thank Professor Qing Zhang for helpful communications and comments. 

\section{The groups and representations}\label{setup}

Let $n, l \in \N$ and $q=p^r$ for some odd prime number $p$. Let $\mathbb{F}_q$ be the finite field of $q$ elements and fix a nontrivial additive character $\psi$ on $\mathbb{F}_q$.

Let $\GL_n$ to be the group $n\times n$ of matrices with entries in $\mathbb{F}_q$ and non-zero determinant. Let $I_n$ be the identity element,

$$J_n=\left(\begin{matrix}
0 & 0 & \cdots & 0 & 1 \\
0 & 0 & \cdots & 1 & 0 \\
\vdots & \vdots & \iddots & \vdots & \vdots \\
0 & 1 & \cdots & 0 & 0 \\
1 & 0 & \cdots & 0 & 0 \\
\end{matrix}\right),$$ and, for $\rho\in\mathbb{F}_q$,
$$J_{2l,\rho}=\mathrm{diag}(I_{l-1}, \left(\begin{matrix}
0 & 1 \\
-\rho & 0
\end{matrix}\right), I_{l-1}) \cdot J_{2l}.$$ We set $\SO_{2n+1}=\{g\in\GL_{2n+1} \, | \, \mathrm{det}(g)=1, {}^tgJ_{2n+1} g = J_{2n+1}\}$ and $\SO_{2l}(J_{2l,\rho})=\{g\in\GL_{2l} \, | \, \mathrm{det}(g)=1, {}^tgJ_{2l,\rho} g = J_{2l,\rho}\}$. We often simply write this as $\SO_{2l}.$ $\SO_{2l}$ is split if $\rho\in(\mathbb{F}_q)^2$ and is quasi-split, but not split, otherwise (in this case $\SO_{2l}$ splits over the quadratic extension $\mathbb{F}_q(\sqrt{\rho})$). Hereinafter, we fix $\SO_{2l}$ to be the quasi-split (non-split) group $\SO_{2l}(J_{2l,\rho})$ for a fixed $\rho\notin(\mathbb{F}_q)^2.$ 
Let $U_{\GL_{n}}$ and $U_{\SO_{2l}}$ be the subgroups of upper triangular matrices in $\GL_n$ and $\SO_{2l}$, respectively. Fix $B_{\SO_{2l}}=T_{\SO_{2l}}U_{\SO_{2l}}$ to be the standard Borel subgroup of $\SO_{2l}$ where $T_{\SO_{2l}}$ consists of all elements of the form
$$
t=\mathrm{diag}(t_1,\dots,t_{l-1},\left(\begin{matrix}
a & b\rho \\
b & a
\end{matrix}\right),t_{l-1}\inv,\dots,t_1\inv),
$$
where $t_1,\dots,t_{l-1}\in \mathbb{F}_q^\times$ and $a,b\in\mathbb{F}_q$ with  $a^2-b^2\rho=1.$ Note that $T_{\SO_{2l}}$ is a maximal torus in $\SO_{2l}.$

Set $$
c=\mathrm{diag}(I_{l-1},
1,-1, I_{l-1}).$$ Note that $c\notin \SO_{2l};$ however, $c \SO_{2l} c\inv=c\SO_{2l}c= \SO_{2l}$. Given a representation $\pi$ of $\SO_{2l},$ we define a new representation $\pi^c$ of $\SO_{2l}$ by $\pi^c(g)=\pi(cgc).$

We now discuss the embeddings needed to define the zeta integrals later. These are the analogues for finite fields of the local cases defined by Kaplan in \cite{Kap13a, Kap15}. If $n<l$ we embed $\SO_{2n+1}$ into $\SO_{2l}$ via
$$
\left(
\begin{matrix}
A & B & C \\
D & E & K \\
L & P & Q 
\end{matrix} 
\right)
\mapsto
\mathrm{diag}(I_{l-n-1},
M^{-1}\left(\begin{matrix}
A & & B & C \\
 & 1 & & \\
D & & E & K \\
L & & P & Q 
\end{matrix}\right)
M, I_{l-n-1})
$$
where $A$ and $Q$ are $n\times n$ matrices and 
$$
M =\mathrm{diag}(I_n,\left(\begin{matrix}
    0 & 2 \\
    1 & 0
    \end{matrix}\right), I_n) 
$$
This embeds $\SO_{2n+1}$ into the standard Levi subgroup $\GL_{l-n-1}\times\SO_{2n+2}$ of $\SO_{2l}$.

Set $\gamma=\frac{\rho}{2}.$ If $l=n$, we embed $\SO_{2l}$ into $\SO_{2l+1}$ via
$$
\left(\begin{matrix}
A & B \\
C & D
\end{matrix}\right)
\mapsto
M^{-1}
\left(\begin{matrix}
A & & B \\
 & 1 & \\
C & & D  
\end{matrix}\right)
M
$$
where $A, B, C,$ and $D$ are $l\times l$ matrices and
$$
M=\mathrm{diag}(I_{l-1},\left(\begin{matrix}
    0 & 1 & 0 \\
    \frac{1}{2} & 0 & \frac{1}{2\gamma} \\
    \frac{1}{2} & 0 & \frac{-1}{2\gamma} 
    \end{matrix}\right), I_{l-1})
$$
Note that the embedding takes the torus $T_{\SO_{2l}}$ to a torus in $\SO_{2l+1}$, but not the standard torus. Indeed, the embedding takes $t$ (as above) to
$$
\mathrm{diag}(s,\left(\begin{matrix}
\frac{1}{2}(1+a) & b & \frac{1}{2\gamma}(1-a) \\
\gamma b & a & -b \\
\frac{\gamma}{2}(1-a) & -\gamma b & \frac{1}{2}(1+a)
\end{matrix}\right), s^*)
$$
where $s=\mathrm{diag}(t_1, t_2, \dots, t_{l-1})$.

Next, we define when a representation is be generic. Recall that $U_{\GL_n}$ and $U_{\SO_{2l}}$ are the subgroups of upper triangular matrices in $\GL_n$ and $\SO_{2l}$ respectively and that we fixed an additive nontrivial character $\psi$ of $\mathbb{F}_q$. We abuse notation and define a generic character, which we will also call $\psi$, on $U_{\GL_n}$ and $U_{\SO_{2l}}$. For $u=(u_{i,j})_{i,j=1}^n\in U_{\GL_n}$, we set $\psi(u)=\psi\left(\sum_{i=1}^{l-1} u_{i,i+1}\right).$ For $u=(u_{i,j})_{i,j=1}^{2l}\in U_{\SO_{2l}},$ we set 
$$
\psi(u)= \psi\left(\sum_{i=1}^{l-2} u_{i,i+1}+\frac{1}{2}u_{l-1,l+1}  \right).
$$

We say an irreducible representation $\pi$ of $\SO_{2l}$ is $\psi$-generic if $$\mathrm{Hom}_{U_{SO_{2l}}}(\pi,\psi)\neq 0.$$
Similarly, we say an irreducible representation $\tau$ of $\GL_{n}$ is $\psi$-generic if $$\mathrm{Hom}_{U_{GL_{n}}}(\tau,\psi)\neq 0.$$
A nonzero intertwining operator in these spaces is called a Whittaker functional and it is well known that Whittaker functionals are unique up to scalars (by uniqueness of Whittaker models). Fix $\Gamma\in\mathrm{Hom}_{U_{SO_{2l}}}(\pi,\psi)$ to be a nonzero Whittaker functional. For $v\in\pi$, let $W_v(g)=\Gamma(\pi(g)v)$ for any $g\in\SO_{2l}$ and set $\mathcal{W}(\pi,\psi)=\{W_v \, | \, v\in\pi\}.$ The space $\mathcal{W}(\pi,\psi)$ is called the $\psi$-Whittaker model of $\pi.$ By Frobenius reciprocity, $\mathrm{Hom}_{U_{SO_{2l}}}(\pi,\psi)\cong\mathrm{Hom}_{SO_{2l}}(\pi,\mathrm{Ind}_{U_{\SO_{2l}}}^{\SO_{2l}}(\psi)).$ Thus, $\pi$ can be realized as a subrepresentation of $\mathrm{Ind}_{U_{\SO_{2l}}}^{\SO_{2l}}(\psi)$ via the map $\pi \rightarrow \mathcal{W}(\pi,\psi)$ given by $v\mapsto W_v.$ Moreover, by uniqueness of Whittaker models, this subrepresentation occurs with multiplicity one inside $\mathrm{Ind}_{U_{\SO_{2l}}}^{\SO_{2l}}(\psi)$. We also note that the analogous results hold for $\psi$-generic representations $\tau$ of $\GL_n.$

Let  $Q_n=L_n V_n$ be the standard Siegel parabolic subgroup of $\SO_{2n+1}$ with Levi subgroup $L_n\cong \GL_n.$ For $a\in\GL_n$ we let $l_n(a)=\mathrm{diag}(a,1,a^*)\in L_n$ where $a^*=J_n{}^ta^{-1}J_n.$ Let $\tau$ be an irreducible $\psi^{-1}$-generic representation of $\GL_{n}$ and set $I(\tau)=\mathrm{Ind}_{Q_n}^{\SO_{2n+1}}\tau.$ An element $\xi\in I(\tau)$ is a function $\xi:\SO_{2n+1}\rightarrow\tau$ satisfying $$
\xi(l_n(a)ug)=\tau(a)\xi(g), \forall a\in\GL_n, u\in V_n, g\in\SO_{2n+1}.
$$
Let $\Lambda_\tau \in \mathrm{Hom}_{U_{GL_{n}}}(\tau,\psi^{-1})$ be a fixed nonzero homomorphism. For $\xi\in I (\tau)$, let $f_\xi : \SO_{2n+1}\times\GL_n \rightarrow \C$ be the function given by $$
f_\xi(g,a)=\Lambda_\tau(\tau(a)\xi(g)).
$$
Let $I(\tau,\psi^{-1})$ be the space of functions generated by $f_\xi, \xi\in I(\tau).$ Note that for $f\in I(\tau,\psi^{-1}),$ we have 
$$
f(g,ua)=\psi^{-1}(u)f(g,a), \forall g\in\SO_{2n+1}, u\in U_{\GL_n}, a\in\GL_n.
$$

We will also let $\tau^*$ be the contragradient representation of $\GL_n$ defined by $\tau^*(a)=\tau(a^*).$

\section{Multiplicity one theorems and the gamma factor}\label{section Multiplicity one theorems and the gamma factor}

The goal of this section is to show that Bessel models for even quasi-split special orthogonal groups over finite fields are unique. The analogous results in the split case are \cite[Propositions 3.1 and 3.2]{HL22a}. The quasi-split case follows from similar arguments. We illustrate the proofs here for convenience. Our primary reference for the setup is \cite{Kap15}. Note that the notation of this section matches with the notation in the setup of \cite{GGP12a} with $H=N^{l-n}\times \SO_{2n+1}$ and $\nu=\psi_\gamma$ in the case $l>n$ and $H=N_0\times\SO_{2l}$ and $\nu=\psi_\gamma$ in the case $l=n$ (in both cases $\psi_\gamma$ is extended to be trivial on the special orthogonal group). Moreover, in the case $l=n,$ $N_0$ and $\nu$ are both trivial. 

\subsection{The case $l=n$}
Assume $l=n.$
In this case, the unipotent piece $N_0$ is trivial and $H$ is $\SO_{2l}$ embedded inside of $\SO_{2l+1}$ as in the previous section.

\begin{prop}\label{MultOne}
Let $Q=LV$ be a parabolic subgroup of $\SO_{2l+1}$ and $\sigma$ be an irreducible representation of $L.$ Let $\pi$ be an irreducible cuspidal representation of $\SO_{2l}$ and $\tau=\mathrm{Ind}_Q^{\SO_{2l+1}} \sigma$. Then

$$
\mathrm{dim}\mathrm{Hom}_{\SO_{2l}}(\pi,\tau)\leq 1.
$$
\end{prop}

\begin{proof}
Note that Proposition 5.1 of \cite{GGP12b} holds in our case. Indeed, we can follow their proof closely, except instead of using the multiplicity one result of \cite{AGRS10}, we use an analogous result for special orthogonal groups from \cite{Wal12}. Thus, the claim holds when $\sigma$ is cuspidal. 

Now suppose that $\sigma$ is not cuspidal. Then there exists a parabolic subgroup $Q^\prime=L^\prime V^\prime$ of $L$ and cuspidal representation $\sigma^\prime$ of $L'$ such that $\sigma\subseteq \mathrm{Ind}_{Q^\prime}^{L}\sigma^\prime.$ By transitivity of parabolic induction, $\tau\subseteq \mathrm{Ind}_{L^\prime V^\prime V}^{\SO_{2l+1}}\sigma^\prime\otimes 1_{ V}$ and hence the claim follows from the case where $\sigma$ is cuspidal.
\end{proof}

\subsection{The case $l> n$}
Assume $l>n$.
We define the subgroup $N^{l-n}\subseteq \SO_{2l}$. Let $P_{l-n-1}=M_{l-n-1}N_{l-n-1}$ be the standard parabolic subgroup of $\SO_{2l}$ with Levi subgroup, $M_{l-n-1}$, isomorphic to $\GL_{l-n-1}\times\SO_{2n+2}.$ We embed $U_{\GL_{l-n-1}}$ inside of $\GL_{l-n-1}$ which is realized inside the Levi subgroup $M_{l-n-1}.$ Define $N^{l-n}=U_{\GL_{l-n-1}}N_{l-n-1}.$ That is,

$$
N^{l-n}=\left\{\left(\begin{matrix}
u_1 & v_1 & v_2 \\
 & I_{{2n+2}} & v_1^\prime \\
 & & u_1^*
\end{matrix}\right) \in \SO_{2l} \ | \ u_1\in U_{\GL_{l-n-1}} \right\}.
$$
For $v=(v_{i,j})_{i,j=1}^{2l}\in N^{l-n},$ we define a character $\psi_\gamma$ of $N^{l-n}$ by

$$
\psi_\gamma(v)= \psi\left(\sum_{i=1}^{l-n-1} v_{i,i+1}+\frac{1}{2}v_{l-n-1,l+1}  \right) . 
$$

Note this character is trivial when $n=l-1.$ Let $H=\SO_{2n+1}N^{l-n}$ where $\SO_{2n+1}$ is realized via the embedding into $\SO_{2n+2}$ (as in \S\ref{setup}) inside $M_{l-n-1}$ and extend $\psi_\gamma$ trivially across $\SO_{2n+1}$ so that $\psi_\gamma$ is now a character of $H$.

\begin{prop}\label{Mult1}
Let $P=MN$ be a parabolic subgroup of $\SO_{2n+1}$ and $\sigma$ be an irreducible representation of the Levi subgroup $M$. Let $\pi$ be an irreducible cuspidal representation of $\SO_{2l}$ and $\tau=\mathrm{Ind}_P^{\SO_{2n+1}} \sigma.$ Then,
$$
\mathrm{dim}\mathrm{Hom}_{H}(\pi,\tau\otimes \psi_\gamma)\leq 1.
$$
\begin{proof}
The proof is the similar to that of Proposition \ref{MultOne}, except instead of using the special orthogonal analogue of \cite[Proposition 5.1]{GGP12b}, we use the special orthogonal analogue of \cite[Proposition 5.3]{GGP12b}.
\end{proof}

\end{prop}

\subsection{The Integrals}

Now, let $\pi$ be an irreducible $\psi$-generic cuspidal representation of $\SO_{2l}$ and $\tau$ be an irreducible $\psi\inv$-generic representation of $\GL_n.$ Let $W\in\mathcal{W}(\pi,\psi)$ and $f\in I(\tau,\psi^{-1})$. Next, we  define ``integrals" $\Psi(W,f)$ analogous to the local integrals of \cite{Kap15}. Over local fields, the corresponding objects are actually integrals; however, over finite fields, these integrals are just sums. Nevertheless, we refer to them as integrals or zeta integrals. Note that \cite{Kap15} defines integrals for any $n$ and $l$; however, we only need the case of $n\leq l$ for the converse theorem and so we do not consider the case $n>l$.

First, suppose that $l=n.$ Then we define
$$
\Psi(W,f)=\sum_{g\in U_{\SO_{2l}}\setminus\SO_{2l}} W(g)f(w_{l,l}g, I_{l})
$$
where 
$$
w_{l,l}=\left(\begin{matrix}
 \gamma I_{l} & &  \\
  & 1 & \\
  & & \gamma\inv I_l
\end{matrix}
\right)\in\SO_{2l+1}.$$
The integral satisfies the property $\Psi(g\cdot W,g\cdot f)=\Psi(W,f)$ for any $g\in\SO_{2l}.$

Next, suppose that $n<l.$ Then we define

$$
\Psi(W,f)=\sum_{g\in U_{\SO_{2n+1}}\setminus\SO_{2n+1}} \left(\sum_{r\in R^{l,n}} W(r w^{l,n} g (w^{l,n})\inv)\right)   f(g, I_{l})
$$
where
$$
w^{l,n}=\left(\begin{matrix}
 & I_n & & & \\
 I_{l-n-1} & & & & \\
 & & I_2 & & \\
 & & & & I_{l-n-1} \\
 & & & I_n &
\end{matrix}
\right)\in\SO_{2l}$$
and
$$
R^{l,n}=\left\{\left(\begin{matrix}
I_n & & & & \\
x & I_{l-n-1} & & & \\
& & I_2 & & \\
& & & I_{l-n-1} & \\
& & & x^\prime & I_n
\end{matrix}\right)\in\SO_{2l}\right\}.
$$
The integral satisfies the property $\Psi((gn)\cdot W,g\cdot f)=\psi\inv_\gamma(n)\Psi(W,f)$ for any $g\in\SO_{2n+1}$ and $n\in N^{l-n}.$ Note that in the case $n<l$, our integral differs from \cite{Kap15} slightly. The difference is a right translation of the Whittaker functional by $(w^{l,n})\inv.$ This is not a significant theoretical issue since right translation preserves $\mathcal{W}(\pi,\psi);$ however, it simplifies the calculations of the integrals later.

Next, we define an intertwining operator $M(\tau,\psi^{-1}): I(\tau,\psi^{-1})\rightarrow I(\tau^*,\psi^{-1})$ by 
$$
M(\tau,\psi^{-1})f(h,a)=\sum_{u\in V_n} f(w_n u h, d_n a^*)
$$
where $w_n=\left(\begin{matrix}
 & & I_n \\
 & (-1)^n & \\
 I_n & &
\end{matrix}\right)$, $d_n=\mathrm{diag}(-1,1,-1,\dots,(-1)^n)\in\GL_n$, and $V_n$ is the unipotent radical of the standard parabolic subgroup $Q_n=L_n V_n$ of $\SO_{2n+1}$ where $L_n\cong \GL_n$, and $a^*=J_n {}^t a^{-1} J_n.$

We now define the $\gamma$-factor in the following proposition.

\begin{prop}\label{gammafactor}
Let $\pi$ be an irreducible $\psi$-generic cuspidal representation of $\SO_{2l}$ and $\tau$ be a $\psi\inv$-generic representation of $\GL_n.$ Let $W\in\mathcal{W}(\pi,\psi)$ and $f\in I(\tau,\psi^{-1})$. Then there exists a constant $\gamma(\pi\times\tau, \psi)\in\mathbb{C}$, called the $\gamma$-factor, such that
$$
\gamma(\pi\times\tau, \psi)\Psi(W,f)=\Psi(W,M(\tau,\psi^{-1})f).
$$
\end{prop}

\begin{proof}
This is immediate from Propositions \ref{MultOne} and \ref{Mult1}.
\end{proof}

\section{Bessel Functions}\label{section Bessel Functions}

In this section, we introduce and study various properties of the Bessel functions. Many of these ideas have local analogues that can be found in \cite{Bar95}. However, over local fields, the situation is more complicated as one has to work with partial Bessel functions as opposed to Bessel functions.

Let $\pi(U_{\SO_{2l}}, \psi)$ be the subspace of $\pi$ generated by $$\{\pi(u)v-\psi(u)v \, | \, u\in U_{SO_{2l}}, \, v\in\pi\}$$ and let $\pi_{U_{\SO_{2l}}, \psi}=\pi/\pi(U_{\SO_{2l}}, \psi)$ be the twisted Jacquet module. Since $\pi$ is an irreducible $\psi$-generic representation, we have $\mathrm{dim}\pi_{U_{\SO_{2l}}, \psi}=1.$ Let $v\in\pi$ such that $v\notin \pi(U_{\SO_{2l}}, \psi)$ and set $$
v_0=\frac{1}{|U_{SO_{2l}}|}\sum_{u\in U_{\SO_{2l}}}\psi^{-1}(u)\pi(u)v.
$$
By the Jacquet-Langlands lemma (see \cite[Lemma 2.33]{BZ76}), $v_0\neq 0.$ Recall that we fixed $\Gamma\in\mathrm{Hom}_{U_{SO_{2l}}}(\pi,\psi)$. We have $\Gamma(v_0)\neq 0$ and by construction it follows that $\pi(u)v_0=\psi(u)v_0$ for any $u\in U_{\SO_{2l}}.$ Any such vector is called a Whittaker vector of $\pi.$
For $g\in\SO_{2l}$, set $$\mathcal{B}_{\pi,\psi}(g)=\frac{\Gamma(\pi(g)v_0)}{\Gamma(v_0)}.$$ $\mathcal{B}_{\pi,\psi}$ is called the normalized Bessel function for $\pi$ (normalized such that the value at the identity is $1$) and it is immediate that $\mathcal{B}_{\pi,\psi}\in\mathcal{W}(\pi,\psi).$ 

\begin{prop}\label{Besselprop} We have
$\mathcal{B}_{\pi,\psi}(I_{2l})=1$ and $\mathcal{B}_{\pi,\psi}(u_1gu_2)=\psi(u_1u_2)\mathcal{B}_{\pi,\psi}(g)$ for any $g\in\SO_{2l}$ and any $u_1,u_2\in U_{\SO_{2l}}.$
\end{prop}

\begin{proof}
This is immediate from the definitions.
\end{proof}

We fix the split part, denoted $S_{\SO_{2l}}$, of the torus $T_{\SO_{2l}}$ to be the subset consisting of elements of the form $t=\mathrm{diag}(t_1,\dots,t_{l-1},I_2,t_{l-1}\inv,\dots,t_1\inv).$ Let $W(\SO_{2l})$ be the Weyl group of $\SO_{2l}$ and let $\Delta(\SO_{2l})$ be the set of simple roots with respect to $S_{\SO_{2l}}$. We can choose the simple roots to be given by $\alpha_i(t)=t_it_{i+1}\inv$ for $i=1,\dots,l-2$ and $\alpha_{l-1}(t)=t_{l-1}.$ In the case of non-split groups, the root subgroups are not guaranteed to be isomorphic to $\mathbb{F}_q$ and so we define elements of $U_{\SO_{2l}}$ that we need here.  Let $x\in\mathbb{F}_q.$ For $i=1,\dots,l-2$ we define $\mathrm{\bold{x}}_{\alpha_i}(x)=(u_{i,j})_{i,j=1}^{2l}\in U_{\SO_{2l}}$ where $u_{j,j}=1,$ for any $j=1,\dots,2l,$ $ u_{i,i+1}=x, u_{2l-i,2l-(i-1)}=-x,$ and $u_{i,j}=0$ for any other pair $(i,j).$ We also define
\begin{align*}
    \mathrm{\bold{x}}_{\alpha_{l-1}}(x)=\left(\begin{matrix} I_{l-1} & & \\
    & \left(\begin{matrix}
        1 & 0 & x &\frac{\rho\inv x^2}{2} \\
        0 & 1 & 0 & 0 \\
        0 & 0 & 1 & \rho\inv x \\
        0 & 0 & 0 & 1
    \end{matrix}\right) & \\
    & & I_{l-1}\end{matrix}\right).
\end{align*}
It is checked directly that $\mathrm{\bold{x}}_{\alpha}(x)\in U_\alpha,$ where $U_\alpha$ denotes the root subgroup of $\alpha,$ for any $\alpha\in\Delta(\SO_{2l}).$

We say that a Weyl element $w\in W(\SO_{2l})$ supports Bessel functions if for any $\alpha\in\Delta(\SO_{2l})$, $w\alpha$ is either negative or simple. We let B($\SO_{2l}$) denote the set of Weyl elements which support Bessel functions. Recall $B_{\SO_{2l}}=T_{\SO_{2l}} U_{\SO_{2l}}$ is a fixed Borel subgroup of $\SO_{2l}.$

\begin{lemma}\label{lemma Bessel support vanish}
Let $\pi$ be an irreducible $\psi$-generic representation of $\SO_{2l}$ with Bessel function $\mathcal{B}_{\pi,\psi}$. Then, for any $w\in W(\SO_{2l})\setminus\mathrm{B}(\SO_{2l}),$ we have $\mathcal{B}_{\pi,\psi}(g)=0,$ for any $g\in B_{\SO_{2l}} w B_{\SO_{2l}}.$
\end{lemma}

\begin{proof}
Since $B_{\SO_{2l}}=T_{\SO_{2l}}U_{\SO_{2l}}=U_{\SO_{2l}}T_{\SO_{2l}}$ as a set,  $$\mathcal{B}_{\pi,\psi}(g)=\mathcal{B}_{\pi,\psi}(t_1u_1wt_2u_2)=\mathcal{B}_{\pi,\psi}(u_1^\prime t_1^\prime w t_2 u_2)$$ for some $u_1,u_1^\prime,u_2\in U_{\SO_{2l}}$ and $t_1,t_1^\prime, t_2\in T_{\SO_{2l}}.$ By Proposition \ref{Besselprop}, we find $\mathcal{B}_{\pi,\psi}(u_1^\prime t_1^\prime w t_2 u_2)=\psi(u_1^\prime u_2)\mathcal{B}_{\pi,\psi}(t_1^\prime w t_2)$ and hence it is enough to show that $\mathcal{B}_{\pi,\psi}(t_1 w t_2)=0$ for any $t_1,t_2\in T_{\SO_{2l}}.$ Since $W(\SO_{2l})=N_{\SO_{2l}}(T_{\SO_{2l}})/T_{\SO_{2l}}$, $wt_2=t_2^\prime w$ for some $t_2^\prime \in T_{\SO_{2l}}.$ Hence, $\mathcal{B}_{\pi,\psi}(t_1 w t_2)=\mathcal{B}_{\pi,\psi}(t_1 t_2^\prime w)$ and so it is enough to show that $\mathcal{B}_{\pi,\psi}(t w)=0$ for any $t\in T_{\SO_{2l}}.$

Now, since $w\notin \mathrm{B}(\SO_{2l}),$ there exists $\alpha\in\Delta(\SO_{2l})$ such that $w\alpha$ is positive but not simple. Let $x\in\mathbb{F}_q$. Then $tw\mathrm{\bold{x}}_\alpha(x) = u'tw$ for some $u'$ in the root subgroup $U_{w\alpha}$. Now, $\psi(\mathrm{\bold{x}}_\alpha(x))$ is a nonzero constant multiple of $\psi(x)$ and this constant is independent of $x$. Since $w\alpha$ is positive but not simple, $\psi(u')=\psi(0)=1.$ By Proposition \ref{Besselprop}, there exists a nonzero constant $c\in\C$ such that $c\psi(x)\mathcal{B}_{\pi,\psi}(t w)=\mathcal{B}_{\pi,\psi}(t w)$ for any $x\in\mathbb{F}_q$. Since $\psi$ is nontrivial, we must have $\mathcal{B}_{\pi,\psi}(t w)=0.$
\end{proof}

The statement of the following lemma is similar to the split case (see \cite[Lemma 4.3]{HL22a}); however, the proof requires more effort here. In the split case, the proof relied on the fact that the intersection of the kernels of the simple roots is precisely the center. This is true in general for split groups, but is not guaranteed for quasi-split groups. To overcome this obstacle, we examine the action of the non-split torus on $\mathrm{\bold{x}}_\alpha(x)$.

\begin{lemma}\label{center}
If $\B_{\pi,\psi}(t)\neq 0$ for some $t\in T_{\SO_{2l}},$ then $t$ is in the center of $\SO_{2l}.$
\end{lemma}

\begin{proof}
Let $t\in T_{\SO_{2l}}$ such that $\B_{\pi,\psi}(t)\neq 0$ and let $\delta\in\Delta(\SO_{2l})\setminus\{\alpha_{l-1}\}.$ For $x\in\mathbb{F}_q$, let $\bold{x}_\delta(x)$ be in the root subgroup of $\delta$. Then $t \bold{x}_\delta(x) = \bold{x}_\delta(\delta(t) x) t.$ Hence, by Proposition \ref{Besselprop}, $$
\psi(\bold{x}_\delta(x))\B_{\pi,\psi}(t) = \psi(\bold{x}_\delta(\delta(t)x))\B_{\pi,\psi}(t).
$$
Thus, $\psi(\bold{x}_\delta(x))= \psi(\bold{x}_\delta(\delta(t)x))$ for any $x\in\mathbb{F}_q$. Since $\psi$ is nontrivial and $x$ is arbitrary, we must have $\delta(t)=1$ for all $\delta\in\Delta(\SO_{2l})\setminus\{\alpha_{l-1}\}.$ Write
$$
t=\mathrm{diag}(t_1,\dots,t_{l-1},\left(\begin{matrix}
a & b\rho \\
b & a
\end{matrix}\right),t_{l-1}\inv,\dots,t_1\inv)
$$
where $a^2-b^2\rho=1.$ Since $\delta(t)=1$ for all $\delta\in\Delta(\SO_{2l})\setminus\{\alpha_{l-1}\},$ we must have $t_1=\cdots=t_{l-1}.$

Next, we examine the root subgroup for $\delta=\alpha_{l-1}.$ We verify directly that 
\begin{align*}
\psi(\bold{x}_\delta(x))\B_{\pi,\psi}(t) = \B_{\pi,\psi}(t \bold{x}_\delta(x)) = \B_{\pi,\psi}(t \bold{x}_\delta(x) t\inv t) &= \B_{\pi,\psi}(t)\psi(t \bold{x}_\delta(x) t\inv) \\ &= \psi(\bold{x}_\delta(at_{l-1}x))\B_{\pi,\psi}(t).
\end{align*}
Similar to the previous case, we deduce that $a=t_{l-1}\inv.$ 
On the other hand, we verify directly
\begin{align*}
\psi(\bold{x}_\delta(x))\B_{\pi,\psi}(t) = \B_{\pi,\psi}(\bold{x}_\delta(x) t) = \B_{\pi,\psi}( t t\inv  \bold{x}_\delta(x)  t) &= \B_{\pi,\psi}(t)\psi(t\inv \bold{x}_\delta(x) t) \\ &= \psi(\bold{x}_\delta(at_{l-1}\inv x))\B_{\pi,\psi}(t).
\end{align*}
Hence $a=t_{l-1}.$ Since $a=t_{l-1}=t_{l-1}\inv$, we have $t_1=t_2=\cdots=t_{l-1}=a=\pm 1$ and $b=0$ which proves the lemma.
\end{proof}

\subsection{Conjugate Bessel Function}

We introduce the Bessel function for $\pi^c$ here.
We have $cU_{\SO_{2l}}c=U_{\SO_{2l}}.$ Hence we define another character of $U_{\SO_{2l}}$ by $\psi_c(u):=\psi(cuc).$ That is, 
$$
\psi_c(u)= \psi\left(\sum_{i=1}^{l-2} u_{i,i+1}-\frac{1}{2}u_{l-1,l+1}  \right),
$$
for any $u=(u_{i,j})_{i,j=1}^{2l}\in U_{\SO_{2l}}.$ Fix $\tilde{t}:=\mathrm{diag}(I_{l-1},-1,-1,I_{l-1})\in T_{\SO_{2l}}.$ It is checked directly that $\psi_c(\tilde{t}\inv u\tilde{t})=\psi(u).$ Recall that we fixed a nonzero Whittaker functional $\Gamma\in\mathrm{Hom}_{U_{\SO_{2l}}}(\pi,\psi).$ Then $\Gamma\in\mathrm{Hom}_{U_{\SO_{2l}}}(\pi^c,\psi_c)$ and $$\mathcal{B}_{\pi^c,\psi_c}(g):=\frac{\Gamma(\pi^c(g)v_0)}{\Gamma(v_0)}=\mathcal{B}_{\pi,\psi}(cgc)$$ defines the normalized Bessel function for $\pi^c$ in $\mathrm{Ind}_{U_{\SO_{2l}}}^{\SO_{2l}}\psi_c$. On the other hand, $\pi^c$ is also $\psi$-generic, since $\psi_c(\tilde{t}\inv u\tilde{t})=\psi(u)$ for any $u\in U_{\SO_{2l}}.$
Let
$$
\mathcal{B}_{\pi^c, \psi}(g):=\mathcal{B}_{\pi^c, \psi_c}(\tilde{t}\inv g \tilde{t})=
\mathcal{B}_{\pi, \psi}(c\tilde{t}\inv g \tilde{t}c).
$$
The notation is justified by the following proposition. 

\begin{prop}\label{Prop Besselconj}
$\mathcal{B}_{\pi^c, \psi}$ is the normalized Bessel function for $\pi^c$ in $\mathrm{Ind}_{U_{\SO_{2l}}}^{\SO_{2l}}\psi.$
\end{prop}

\begin{proof}
The proof is similar to \cite[Proposition 4.4]{HL22a}. 
\end{proof}

\subsection{Partition of the Bessel Support}

The maximal standard parabolic subgroups of  $\SO_{2l}$ are $P_n=M_n N_n$ for $n\leq l-1$ where $M_n$ is the standard Levi subgroup that is isomorphic to $\GL_n\times\SO_{2(l-n)}$. Let $w_{long}$ be the long Weyl element of $\SO_{2l}$ and $w_{M_n}$ be the long Weyl element in $M_n.$ Define $\tilde{w}_n=w_{long}\inv w_{M_n}.$ The analogues of these elements played an essential role in the arguments for the other classical groups \cite{HL22a, LZ22a} and they continue to do so in our case. To be explicit, we have
$$
w_{long}=\left(\begin{matrix}
 & & J_{l-1} \\
 & \left(\begin{matrix}
 -1 & 0 \\
 0 & 1 
 \end{matrix}\right)^{l-1} & \\
 J_{l-1} & &
\end{matrix}\right), 
w_{M_n}=\left(\begin{matrix}
J_n & &  \\
 & w_{long,2(l-n)} & \\
  & & J_n
\end{matrix}\right),$$
where $w_{long,2(l-n)}$ is the long Weyl element of $\SO_{2(l-n)}.$
Hence,
$$
\tilde{w}_n=\left(\begin{matrix}
& & & & I_n \\
 & I_{l-n-1} &  & & \\
 & & \left(\begin{matrix}
 -1 & 0 \\
 0 & 1 
 \end{matrix}\right)^{n} & & & \\
 & & & I_{l-n-1} &  \\
 I_n & & & &
\end{matrix}\right).
$$
Recall that the definition of the intertwining operator $M(\tau,\psi\inv)$ involved the Weyl element $w_n$ of $\SO_{2n+1}.$ The image of $w_n$ under the embedding of $\SO_{2n+1}$ into $\SO_{2l}$ and conjugated by $w^{l,n}$ is $\tilde{w}_n.$ That is, $w^{l,n} w_n (w^{l,n})^{-1}=\tilde{w}_n.$

Let $t$ be in the split part of the torus and write
$$
t=\mathrm{diag}(t_1,\dots,t_{l-1},I_2,t_{l-1}\inv,\dots,t_1\inv).
$$ 
So 
$
\tilde{w}_n\inv t \tilde{w}_n=\mathrm{diag}(t_n\inv,\dots,t_1\inv,t_{n+1},\dots,t_{l-1},I_2,t_{l-1}\inv,\dots,t_{n+1}\inv,t_1,\dots,t_n).
$
Thus, we have
$$\tilde{w}_n\alpha_i = \alpha_{n-i} \, \, \mathrm{for} \, \, 1\leq i \leq n-1,$$
$$\tilde{w}_n\alpha_n (t) = t_1\inv t_{n+1}\inv,$$ 
$$\tilde{w}_n\alpha_i=\alpha_i \, \, \mathrm{for} \, \, n+1\leq i \leq l-1.$$ 

Let $\theta_w=\{\alpha\in\Delta(\SO_{2l}) \, | \, w \alpha\in\Phi^+(\SO_{2l})\}.$ The assignment $w\mapsto \theta_w$ gives a bijection from $\mathrm{B}(\SO_{2l})$ to the power set of $\Delta(\SO_{2l})$, which we denote by $\mathcal{P}(\Delta(\SO_{2l}))$. Then, for $n\leq l-1,$ $\theta_{\tilde{w}_n}=\Delta(\SO_{2l})\setminus \{\alpha_n\}$.

For $n\leq l-1,$ let $\mathrm{B}_n(\SO_{2l})$ be the set of $w\in \mathrm{B}(\SO_{2l})$ such that there exists $w'\in W(\GL_n)$ such that $w=t_n(w')\tilde{w}_n.$ We let $\mathrm{B}_0(\SO_{2l})$ be the identity. Also, let 
$$
P_n=\{ \theta\subseteq\Delta(\SO_{2l}) \, | \, w_\theta\in \mathrm{B}_n(\SO_{2l})\}.
$$

\begin{prop}\label{Besselprtnon}
$$P_n=\{ \theta\subseteq\Delta(\SO_{2l}) \, | \, \{\alpha_{n+1},\dots,\alpha_{l-1}\} \subseteq \theta \subseteq \Delta(\SO_{2l})\setminus \{\alpha_n\}\}.$$
\end{prop}

\begin{proof}
This is implicitly proved in \cite[\S 5, 7]{LZ22a}. Indeed, we do not need to refer to the matrices explicitly. If $w=t_n(w')\tilde{w}_n$ for some $w'\in W(\GL_n),$ then $t_n(w')\alpha_i=\alpha_i$ for $i\geq n+1$. Also, $t_n(w')$ acts on the first $n$ coordinates of $\tilde{w}_n\inv t\tilde{w}_n$ by permuting the $t_i\inv$'s for $i\leq n.$ Hence $w\alpha_n(t)=t_i\inv t_{n+1}\inv$ for some $i\leq n$ and is hence negative. Thus,
$$
P_n\subseteq\{ \theta\subseteq\Delta(\SO_{2l}) \, | \, \{\alpha_{n+1},\dots,\alpha_{l-1}\} \subseteq \theta \subseteq \Delta(\SO_{2l})\setminus \{\alpha_n\}\}.$$
The reverse inclusion is straightforward to check.
\end{proof}

\begin{cor}\label{Besselpartnonsplit} The sets $\mathrm{B}_n(\SO_{2l})$ for $n\leq l-1$ form a partition of the Bessel support $\mathrm{B}(\SO_{2l}).$
\end{cor}

\begin{proof}
Again, this is implicit in \cite[\S 5, 7]{LZ22a}. The map $w\mapsto\theta_w$ gives an bijection from $\mathrm{B}(\SO_{2l})$ to $\mathcal{P}(\Delta(\SO_{2l}))$. By Proposition \ref{Besselprtnon}, the sets $P_n$ such that $n\leq l-1$ form a partition of $\mathcal{P}(\Delta(\SO_{2l})).$ From the bijection $w\mapsto\theta_w$, we have that the sets $\mathrm{B}_n(\SO_{2l})$ for $n\leq l-1$ form a partition of the Bessel support $\mathrm{B}(\SO_{2l}).$
\end{proof}

The following proposition is used later to show that the zeta integrals are nonzero.

\begin{prop}\label{uppertriangular}
If $w\in W(\GL_{l-1})$ and $w\neq I_{l-1}$, then $t_{l-1}(w)\notin \mathrm{B}(\SO_{2l}).$ In particular, if $a\in\GL_{l-1}$ and $a$ is not upper triangular, then $\B_{\pi,\psi}(t_{l-1}(a))=0.$
\end{prop}

\begin{proof}
We use the partition of the Bessel support in Corollary \ref{Besselpartnonsplit}.
First, suppose that $t_{l-1}(w)=t_n(w')\tilde{w}_n$ for some $n$ with $1\leq n\leq {l-1}$ and $w'\in W(\GL_n).$ Then, $\tilde{w}_n\in W(\GL_{l-1}).$ However, this is not the case for any $n$. Thus we have $t_{l-1}(w)\notin \mathrm{B}(\SO_{2l}).$

For the second part of the claim, suppose $a\in\GL_{l-1}$ is not upper triangular, Then, by the Bruhat decomposition of $\GL_{l-1}$, there exists $w\in W(\GL_{l-1})$ such that $w\neq I_{l-1}$ and $a=u_1t w u_2$ where $t\in T_{\GL_{l-1}}$ and $u_1, u_2\in U_{\GL_{l-1}}.$ We obtain $\B_{\pi,\psi}(t_{l-1}(a))=\B_{\pi,\psi}(t_{l-1}(t)t_{l-1}(w))\psi(t_{l-1}(u_1))\psi(t_{l-1}(u_2)).$ However, $t_{l-1}(w)\notin \mathrm{B}(\SO_{2l})$ and hence $\B_{\pi,\psi}(t_{l-1}(t)t_{l-1}(w))=0$ by Lemma \ref{lemma Bessel support vanish}.
\end{proof}

The next proposition concerns the support of the Bessel functions on Bruhat cells associated to $w\in\mathrm{B}_n(\SO_{2l}).$ The argument is similar to that of Lemma \ref{center}.

\begin{prop}\label{prop support on B_n}
Let $t\in T_{\SO_{2l}}$ and $w\in\mathrm{B}_n(\SO_{2l})$ for $n< l-1.$ Suppose that $\B_{\pi,\psi}(tw)\neq 0.$ If we write
$$
t=\mathrm{diag}(t_1,\dots,t_{l-1},\left(\begin{matrix}
a & b\rho \\
b & a
\end{matrix}\right),t_{l-1}\inv,\dots,t_1\inv)
$$
where $a^2-b^2\rho=1,$ then we must have $a=t_{l-1}=\cdots=t_{n+1}=\pm1$ and $b=0.$
\end{prop}

\begin{proof}
By definition, there exists $w'\in W(\GL_n)$ such that $w=t_n(w')\tilde{w}_n.$ We verify directly that $w\alpha_i=\alpha_i$ for $i=n+1,\dots,l-1.$ Let $\delta\in\{\alpha_{n+1},\dots,\alpha_{l-2}\}.$ For $x\in\mathbb{F}_q$, let $\bold{x}_\delta(x)$ be in the root subgroup of $\delta$. Then $tw \bold{x}_\delta(x) = \bold{x}_\delta(\delta(t) x) tw.$ Hence, by Proposition \ref{Besselprop}, $$
\psi(\bold{x}_\delta(x))\B_{\pi,\psi}(tw) = \psi(\bold{x}_\delta(\delta(t)x))\B_{\pi,\psi}(tw).
$$
Thus, $\psi(\bold{x}_\delta(x))= \psi(\bold{x}_\delta(\delta(t)x))$ for any $x\in\mathbb{F}_q$. Since $\psi$ is nontrivial and $x$ is arbitrary, we must have $t_{n+1}=\cdots=t_{l-2}.$

Next, we examine the root subgroup for $\delta=\alpha_{l-1}.$ We verify directly that 
\begin{align*}
\psi(\bold{x}_\delta(x))\B_{\pi,\psi}(tw) = \B_{\pi,\psi}(tw \bold{x}_\delta(x)) = \B_{\pi,\psi}(t \bold{x}_\delta(x) t\inv tw) &= \B_{\pi,\psi}(tw)\psi(t \bold{x}_\delta(x) t\inv) \\ &= \psi(\bold{x}_\delta(at_{l-1}x))\B_{\pi,\psi}(tw).
\end{align*}
Similar to the previous case, we deduce that $a=t_{l-1}\inv.$ 
On the other hand, we verify directly
\begin{align*}
\psi(\bold{x}_\delta(x))\B_{\pi,\psi}(tw) = \B_{\pi,\psi}(\bold{x}_\delta(x) tw) = \B_{\pi,\psi}( tw t\inv  \bold{x}_\delta(x)  t) &= \B_{\pi,\psi}(tw)\psi(t\inv \bold{x}_\delta(x) t) \\ &= \psi(\bold{x}_\delta(at_{l-1}\inv x))\B_{\pi,\psi}(tw).
\end{align*}
Hence $a=t_{l-1}.$ Since $a=t_{l-1}=t_{l-1}\inv$, $t_{n+1}=\cdots=t_{l-1}=a=\pm 1$ and $b=0$ which proves the lemma.
\end{proof}

\section{Twists by $\GL_n$ for $n\leq l-2$}\label{section Twists by GLn}

We begin by collecting some results which will be used in determining the $\GL_n$-twists. The first result is a mild strengthening of \cite[Lemma 3.1]{Nie14}.

\begin{lemma}[{\cite[Lemma 5.1]{HL22a}}]\label{Niennon}
Let $H$ be function on $\GL_n$ such that $H(ug)=\psi(u)H(g)$ for any $u\in U_{\GL_n}$ and $g\in\GL_n.$
Let $X$ be a subset of $\GL_n$ and suppose that 
$$\sum_{x\in X} H(x)W_v(x)=0
$$
for any $v\in\tau$ and with $\tau$ running through all $\psi\inv-$generic irreducible representations of $\GL_n.$ Then, $H(x)=0$ for any $x\in X.$
\end{lemma}

Now we start the calculations of the zeta integrals for twists up to $\GL_{l-2}.$
Let $\tau$ be an irreducible $\psi^{-1}$-generic representation of $\GL_n$, $v\in\tau$ be a fixed vector, and define $\xi_v\in I(\tau)$ by supp$(\xi_v)=Q_n$ and 
$$
\xi_v(l_n(a)u)=\tau(a)v, \forall a\in\GL_n, u\in V_n,
$$
where $Q_n=L_n V_n$ is the standard Siegel parabolic of $\SO_{2n+1}.$
We define the function $f_v=f_{\xi_v}\in I(\tau,\psi^{-1}).$ That is, $f_v(g,a)=\Lambda_\tau(\tau(a)\xi_v(g))$ where we have $\Lambda_\tau\in \mathrm{Hom}_{U_{GL_{n}}}(\tau,\psi^{-1}),$ $a\in\GL_n,$ and $g\in\SO_{2n+1}.$ We also fix the Whittaker function $W_v(a)=\Lambda_\tau(\tau(a)v).$

\begin{prop}\label{GLnnonzero}
Suppose that $n\leq l-2$ and $\B_{\pi,\psi}$ is the normalized Bessel function of $\pi$, an irreducible cuspidal $\psi$-generic representation of $\SO_{2l}$. Then $\Psi(\B_{\pi,\psi},f_v)= W_v(I_n)$. Moreover, we may choose $v\in\tau$ such that $\Psi(\B_{\pi,\psi},f_v)\neq 0.$
\end{prop}

\begin{proof}
By definition, $$\Psi(\B_{\pi,\psi},f_v)
=\sum_{g\in U_{\SO_{2n+1}}\setminus\SO_{2n+1}} \left(\sum_{r\in R^{l,n}} \B_{\pi,\psi}(r w^{l,n} g (w^{l,n})\inv)\right)   f_v(g, I_{l}).
$$
The support of $f_v(g, I_n)$ is contained in $V_n L_n$. Since $V_n\subseteq U_{\SO_{2n+1}}$, we have 
$$\Psi(\B_{\pi,\psi},f_v)
=\sum_{a\in U_{GL_n}\setminus\GL_n} \left(\sum_{r\in R^{l,n}} \B_{\pi,\psi}(r w^{l,n} l_n(a) (w^{l,n})\inv)\right)   W_v(a).
$$
Now, the embedding of $\SO_{2n+1}$ into $\SO_{2l}$ takes $l_n(a)$ to $q_n(a)=\mathrm{diag}(I_{l-n-1},a,I_2,a^*,I_{l-n-1}).$ Thus, $w^{l,n} q_n(a) (w^{l,n})\inv=\mathrm{diag}(a, I_{2l-2n}, a^*).$ Let $$r_x=\left(\begin{matrix}
I_n & & & & \\
x & I_{l-n-1} & & & \\
& & I_2 & & \\
& & & I_{l-n-1} & \\
& & & x^\prime & I_n
\end{matrix}\right).$$ Then, $$
r_x w^{l,n} q_n(a) (w^{l,n})\inv
=\left(\begin{matrix}
a & & & & \\
xa & I_{l-n-1} & & & \\
& & I_2 & & \\
& & & I_{l-n-1} & \\
& & & x^\prime a^* & a^*
\end{matrix}\right).
$$
By Proposition \ref{uppertriangular}, $\B_{\pi,\psi}(r_x w^{l,n} q_n(a) (w^{l,n})\inv)=0$ unless $a$ is upper triangular and $x=0.$ Hence, 
$$\Psi(\B_{\pi,\psi},f_v)
=\sum_{a\in T_{\GL_n}} \B_{\pi,\psi}(t_n(a))   W_v(a).
$$
Finally, by Lemma \ref{center}, $\Psi(\B_{\pi,\psi},f_v)
=W_v(I_n).$ Finally, if we choose $v$ to be a Whittaker vector for $\tau$, then $W_v(I_n)\neq 0.$
\end{proof}

Let $\tilde{f}_v=M(t,\psi\inv)f_v$ and $W_v^*(a)=\Lambda_\tau(\tau(d_n a^*)v).$ The following lemma lists several properties of $\tilde{f}_v$ that will be used in computing the zeta integrals.
\begin{lemma}[{\cite[Lemma 6.2]{HL22a}}]\label{intertwinenon}
For any $n\leq l$, we have the following:
\begin{enumerate}
    \item If $\tilde{f}_v(g, I_n)\neq 0,$ then $g\in Q_n w_n Q_n = Q_n w_n V_n.$
    \item If $x\in V_n$, then $\tilde{f}_v(w_n x, I_n)=W^*_v(I_n).$
    \item For $a\in\GL_n$ and $x\in V_n,$ $\tilde{f}_v(l_n(a)w_n x, I_n)=W_v^*(a).$
\end{enumerate}
\end{lemma}

The following theorem is the main result for twists up to $\GL_{l-2}.$

\begin{thm}\label{GL_n}
Suppose that $n\leq l-2.$ Let $\pi$ and $\pi^\prime$ be irreducible cuspidal $\psi$-generic representations of $\SO_{2l}$ which share the same central character. If $\gamma(\pi\times\tau,\psi)=\gamma(\pi^\prime\times\tau,\psi)$ for all irreducible $\psi^{-1}$-generic representations $\tau$ of $\GL_n$, then
$$\B_{\pi,\psi}(t_n(a)  \tilde{w_n})=\B_{\pi^\prime,\psi}(t_n(a)   \tilde{w_n})$$ for any $a\in\GL_n.$
\end{thm}

\begin{proof}
By Proposition \ref{GLnnonzero}, $\Psi(\B_{\pi,\psi},f_v)=\Psi(\B_{\pi^\prime,\psi},f_v)=W_v(I_n).$ By assumption, $\gamma(\pi\times\tau,\psi)=\gamma(\pi^\prime\times\tau,\psi)$ for all irreducible $\psi^{-1}$-generic representations $\tau$ of $\GL_n$, and hence it follows that $\Psi(\B_{\pi,\psi},\tilde{f_v})=\Psi(\B_{\pi^\prime,\psi},\tilde{f_v}).$ By definition,
$$\Psi(\B_{\pi,\psi},\tilde{f}_v)
=\sum_{g\in U_{\SO_{2n+1}}\setminus\SO_{2n+1}} \left(\sum_{r\in R^{l,n}} \B_{\pi,\psi}(r w^{l,n} g (w^{l,n})\inv)\right)   \tilde{f}_v(g, I_{n}).
$$
By Lemma \ref{intertwinenon},
$$\Psi(\B_{\pi,\psi},\tilde{f}_v)
=\sum_{\substack{a\in U_{GL_n}\setminus\GL_{n} \\ x\in V_n}} \left(\sum_{r\in R^{l,n}} \B_{\pi,\psi}(r w^{l,n} l_n(a) w_n x (w^{l,n})\inv)\right)   \tilde{f}_v(l_n(a) w_n x, I_{n}).
$$

Now, the embedding of $\SO_{2n+1}$ into $\SO_{2l}$ takes  $l_n(a)$ to $q_n(a)=\mathrm{diag}(I_{l-n-1}, a, I_2, a^*, I_{l-n-1}).$
The embedding also takes the unipotent element $x=\left(\begin{matrix}
I_n & * & * \\
& 1 & * \\
& & I_n
\end{matrix}\right)$ to $\tilde{x}=\mathrm{diag}(I_{l-n-1},\left(\begin{matrix}
I_n & * & * \\
& I_2 & * \\
& & I_n
\end{matrix}\right), I_{l-n-1}).$ Now,
$$
\tilde{x}(w^{l,n})\inv=(w^{l,n})\inv\left(\begin{matrix}
I_n & 0 &  * & 0 & * \\
& I_{l-n-1} & 0 & 0 & 0\\
& & I_2 & 0 & * \\
& & & I_{l-n-1} & 0 \\
& & & & I_n
\end{matrix}\right).
$$
The character is trivial on the last unipotent matrix and hence, by Lemma \ref{intertwinenon},
$$\Psi(\B_{\pi,\psi},\tilde{f}_v)
=|V_n| \sum_{a\in U_{GL_n}\setminus\GL_{n}} \left(\sum_{r\in R^{l,n}} \B_{\pi,\psi}(r w^{l,n} q_n(a) w_n (w^{l,n})\inv)\right)   W_v^*(a).
$$

Next, $w^{l,n} q_n(a)=t_n(a) w^{l,n}$ where $t_n(a)=\mathrm{diag}(a,I_{2l-2n}, a^*).$ Let $$r_x=\left(\begin{matrix}
I_n & & & & \\
x & I_{l-n-1} & & & \\
& & I_2 & & \\
& & & I_{l-n-1} & \\
& & & x^\prime & I_n
\end{matrix}\right).$$ Then, $r_x t_n(a) =t_n(a) r_{ax}.$ Recall that $\tilde{w}_n=w^{l,n}w_n(w^{l,n})\inv$ where $w_n\in\SO_{2n+1}$ is realized under the embedding into $\SO_{2l}.$ Then,
$$\Psi(\B_{\pi,\psi},\tilde{f}_v)
=|V_n| \sum_{a\in U_{GL_n}\setminus\GL_{n}} \left(\sum_{r_x\in R^{l,n}} \B_{\pi,\psi}(t_n(a)  r_{ax} \tilde{w}_n)\right)   W_v^*(a).
$$
Now, $$r_{ax}\tilde{w}_n=\tilde{w}_n\left(\begin{matrix}
I_n & &&  x'(a^*) & \\
& I_{l-n-1} & & & xa \\
& & I_2 & & \\
& & & I_{l-n-1} & \\
& & & & I_n
\end{matrix}\right).$$ The character is trivial on the last unipotent element and hence
$$\Psi(\B_{\pi,\psi},\tilde{f}_v)
=|V_n||R^{l,n}| \sum_{a\in U_{GL_n}\setminus\GL_{n}}  \B_{\pi,\psi}(t_n(a)  \tilde{w}_n)   W_v^*(a).
$$
Since $\Psi(\B_{\pi,\psi},\tilde{f_v})=\Psi(\B_{\pi^\prime,\psi},\tilde{f_v}),$ it follows that $$0=\sum_{a\in U_{GL_n}\setminus\GL_{n}}  (\B_{\pi,\psi}-\B_{\pi',\psi})(t_n(a)   \tilde{w}_n)   W_v^*(a).$$

Now, let $f$ be the function on $\GL_n$ defined by $f(a)=(\B_{\pi,\psi}-\B_{\pi',\psi})(t_n(a)   \tilde{w}_n).$ Then, $f(ua)=\psi(u)f(a)$ and $$0=\sum_{a\in U_{GL_n}\setminus\GL_{n}}  f(a)   W_v^*(a).$$ By Lemma \ref{Niennon}, $f(a)=0$ for any $a\in\GL_n$ which gives the claim.
\end{proof}

\section{Twists by $\GL_{l-1}$}\label{section Twists by GL l-1}

In this section, we consider the twists by $\GL_{l-1}$. The arguments in this section are similar to the previous section, except that $R^{l,n}$ and $w^{l,n}$ are trivial which leads to some minor differences.

Let $\tau$ be an irreducible $\psi^{-1}$-generic representation of $\GL_{l-1}$, $v\in\tau$ be a fixed vector, and define $\xi_v\in I(\tau)$ by supp$(\xi_v)=L_{l-1} V_{l-1}= Q_{l-1}$ and 
$$
\xi_v(l_{l-1}(a)u)=\tau(a)v, \forall a\in\GL_{l-1}, u\in V_{l-1}.
$$
Let $f_v=f_{\xi_v}\in I(\tau,\psi^{-1}).$ That is, for $\Lambda_\tau\in \mathrm{Hom}_{U_{GL_{l-1}}}(\tau,\psi^{-1}), a\in\GL_{l-1},$ and $g\in\SO_{2l-1}$, we have $f_v(g,a)=\Lambda_\tau(\tau(a)\xi_v(g))$. We also let $W_v(a)=\Lambda_\tau(\tau(a)v).$ The zeta integrals in this case are nonzero. 

\begin{prop}\label{GL_{l-1}nonzero}
Suppose $\B_{\pi,\psi}$ be the normalized Bessel function of $\pi$, an irreducible cuspidal $\psi$-generic representation of $\SO_{2l}$. Then, we have
$\Psi(\B_{\pi,\psi},f_v)= W_v(I_{l-1})$. In particular, we may choose $v\in\tau$ such that $\Psi(\B_{\pi,\psi},f_v)\neq 0.$
\end{prop}

\begin{proof}
By definition, $$\Psi(\B_{\pi,\psi},f_v)=\sum_{g\in U_{\SO_{2l-1}}\setminus \SO_{2l-1}} \B_{\pi,\psi}(g) f_v(g, I_{l-1}).$$
The support of $f_v(\cdot,I_{l-1})$ is $Q_{l-1}.$ Thus, $$\Psi(\B_{\pi,\psi},f_v)=\sum_{a\in U_{\GL_{l-1}}\setminus \GL_{l-1}} \B_{\pi,\psi}(l_{l-1}(a)) f_v(l_{l-1}(a), I_{l-1}).$$
Under the embedding of $\SO_{2l-1}$ into $\SO_{2l}$, $l_{l-1}(a)$ maps to $t_{l-1}(a).$ Hence, 
$$\Psi(\B_{\pi,\psi},f_v)=\sum_{a\in U_{\GL_{l-1}}\setminus \GL_{l-1}} \B_{\pi,\psi}(t_{l-1}(a)) W_v(a).$$
By Proposition \ref{uppertriangular}, $\B_{\pi,\psi}(t_{l-1}(a))=0$ unless $a$ is upper triangular. Thus, by Lemma $\ref{center}$,
$$\Psi(\B_{\pi,\psi},f_v)=W_v(I_{l-1}).$$ If we choose $v$ to be a Whittaker vector for $\tau$, then $W_v(I_{l-1})\neq 0.$ This proves the proposition.
 \end{proof}

The next theorem shows that the twists by $\GL_{l-1}$ determine that the Bessel functions are equal on a subset of the Bruhat cells in $\mathrm{B}_{l-1}(\SO_{2l}).$ More specifically, the Bessel functions are determined on the parts of the cell corresponding to the split part of the torus.

\begin{thm}\label{GL_{l-1}}
Let $\pi$ and $\pi^\prime$ be irreducible cuspidal $\psi$-generic representations of $\SO_{2l}$ which share the same central character. If $\gamma(\pi\times\tau,\psi)=\gamma(\pi^\prime\times\tau,\psi)$ for all irreducible generic representations $\tau$ of $\GL_{l-1}$, then we have
$$\B_{\pi,\psi}(t_{l-1}(a) \tilde{w}_{l-1})=\B_{\pi^\prime,\psi}(t_{l-1}(a)  \tilde{w}_{l-1})$$ for any $a\in\GL_{l-1}.$
\end{thm}

\begin{proof}
By Proposition \ref{GL_{l-1}nonzero}, $\Psi(\B_{\pi,\psi},f_v)=\Psi(\B_{\pi^\prime,\psi},f_v)=W_v(I_{l-1}).$ From our assumption, $\gamma(\pi\times\tau,\psi)=\gamma(\pi^\prime\times\tau,\psi)$ for all irreducible generic representations $\tau$ of $\GL_{l-1}$, and hence it follows that $\Psi(\B_{\pi,\psi},\tilde{f_v})=\Psi(\B_{\pi^\prime,\psi},\tilde{f_v}).$ By definition,
$$\Psi(\B_{\pi,\psi},\tilde{f}_v)=\sum_{g\in U_{\SO_{2l-1}}\setminus \SO_{2l-1}} \B_{\pi,\psi}(g) \tilde{f}_v(g, I_{l-1}).$$
From Lemma \ref{intertwinenon}, we have
$$\Psi(\B_{\pi,\psi},\tilde{f}_v)=\sum_{\substack{a\in U_{\GL_{l-1}}\setminus \GL_{l-1} \\ u\in V_{l-1}}} \B_{\pi,\psi}(l_{l-1}(a)w_{l-1} u) W_v^*(a).$$
The embedding of $\SO_{2l-1}$ into $\SO_{2l}$ takes $l_{l-1}(a)$ to $t_{l-1}(a)$ and $w_{l-1}$ to $\tilde{w}_{l-1}.$ Let $$
u=\left(\begin{matrix}
I_{l-1} & x & y \\
& 1 & x' \\
& & I_{l-1}
\end{matrix}\right).
$$
Then the embedding takes $u$ to
$$
\tilde{u}=\left(\begin{matrix}
I_{l-1} & x &  & y \\
& 1 & & x' \\
& & 1 &  \\
& & & I_{l-1}
\end{matrix}\right).
$$
Since $\psi(\tilde{u})=1,$ we have 
$$\Psi(\B_{\pi,\psi},\tilde{f}_v)=|V_{l-1}|\sum_{a\in U_{\GL_{l-1}}\setminus \GL_{l-1} } \B_{\pi,\psi}(t_{l-1}(a)\tilde{w}_{l-1} ) W_v^*(a).$$
From $\Psi(\B_{\pi,\psi},\tilde{f_v})=\Psi(\B_{\pi^\prime,\psi},\tilde{f_v}),$ it follows that $$0=\sum_{a\in U_{GL_{l-1}}\setminus\GL_{l-1}}  (\B_{\pi,\psi}-\B_{\pi',\psi})(t_{l-1}(a)   \tilde{w}_{l-1})   W_v^*(a).$$

Let $f$ be the function on $\GL_{l-1}$ defined by $f(a)=(\B_{\pi,\psi}-\B_{\pi',\psi})(t_{l-1}(a)  \tilde{w}_{l-1}).$ Then, $f(ua)=\psi(u)f(a)$ and $$0=\sum_{a\in\GL_{l-1}}  f(a)   W_v^*(a).$$ By Lemma \ref{Niennon}, $f(a)=0$ for any $a\in\GL_{l-1}$ which gives the theorem.
 \end{proof}

The following corollary shows that Theorems \ref{GL_n} and \ref{GL_{l-1}} hold for the Bessel functions of $\pi^c$ and $\pi'{}^c$ as well.

\begin{cor}\label{cGL_n}
Let $\pi$ and $\pi^\prime$ be irreducible cuspidal $\psi$-generic representations of $\SO_{2l}$ which share the same central character. If $\gamma(\pi\times\tau,\psi)=\gamma(\pi^\prime\times\tau,\psi)$ for all irreducible generic representations $\tau$ of $\GL_{n}$ for any $n\leq l-1$, then
$\B_{\pi^c,\psi}(t_{n}(a)   \tilde{w}_{n})=\B_{\pi'{}^c,\psi}(t_{n}(a)   \tilde{w}_{n})$ for any $a\in\GL_{n}.$
\end{cor}

\begin{proof}
By Theorems \ref{GL_n} and \ref{GL_{l-1}}, we have
$\B_{\pi,\psi}(t_{n}(a) \tilde{w}_{n})=\B_{\pi^\prime,\psi}(t_{n}(a)  \tilde{w}_{n})$ for any $a\in\GL_{n}$  and $n\leq l-1.$ Also, 
\begin{align*}
    \B_{\pi,\psi}(t_{n}(a)   \tilde{w}_{n})
=\B_{\pi,\psi}(c\tilde{t}\inv \tilde{t} ct_{n}(a)   \tilde{w}_{n} c\tilde{t}\inv \tilde{t} c)
&=\B_{\pi^c,\psi}( \tilde{t} ct_{n}(a)  \tilde{w}_{n} c\tilde{t}\inv) \\
&=\B_{\pi^c,\psi}( t_{n}(a)  \tilde{w}_{n}).
\end{align*}
This proves the corollary.
 \end{proof}

The following corollary shows the equivalence of $\gamma$-factors of conjugate representations for $n\leq l-1$. This is not necessary to show the converse theorem, but when later paired with its analogue for $n=l$, Corollary \ref{conj l gamma}, it shows that the $\gamma$-factor cannot distinguish between a representation and its conjugate. In the setting of local fields, this is already known (\cite{Art13, JL14}).

\begin{cor}\label{conj n gamma}
Let $\pi$ be an irreducible cuspidal $\psi$-generic representation of $\SO_{2l}$. Then we have $\gamma(\pi\times\tau,\psi)=\gamma(\pi^c\times\tau,\psi)$ for all irreducible generic representations $\tau$ of $\GL_{n}$ for any $n\leq l-1$.
\end{cor}

\begin{proof}
We abuse notation and consider the zeta integrals simultaneously for each $n\leq l-1.$ By Propositions \ref{GLnnonzero} and \ref{GL_{l-1}nonzero} we can choose $v\in\tau$ such that $\Psi(\B_{\pi,\psi},f_v)=\Psi(\B_{\pi^c,\psi},f_v)\neq 0.$ By the proofs of Theorems \ref{GL_n} and \ref{GL_{l-1}},
$$\Psi(\B_{\pi,\psi},\tilde{f}_v)=|V_{n}||R^{l,n}|\sum_{a\in U_{\GL_{n}}\setminus \GL_{n} } \B_{\pi,\psi}(t_{n}(a)\tilde{w}_{n} ) W_v^*(a).$$ Since for any $n\leq l-1,$ we have
\begin{align*}
    \B_{\pi,\psi}(t_{n}(a)   \tilde{w}_{n})
&=\B_{\pi,\psi}(c\tilde{t}\inv \tilde{t} ct_{n}(a)   \tilde{w}_{n} c\tilde{t}\inv \tilde{t} c)\\
&=\B_{\pi^c,\psi}( \tilde{t} ct_{n}(a)   \tilde{w}_{n} c\tilde{t}\inv)\\
&=\B_{\pi^c,\psi}( t_{n}(a)   \tilde{w}_{n}),
\end{align*}
it follows that $\Psi(\B_{\pi,\psi},\tilde{f}_v)=\Psi(\B_{\pi^c,\psi},\tilde{f}_v).$ Finally, since $\Psi(\B_{\pi,\psi},f_v)=\Psi(\B_{\pi^c,\psi},f_v)$ is nonzero, we have $\gamma(\pi\times\tau,\psi)=\gamma(\pi^c\times\tau,\psi)$ for all irreducible generic representations $\tau$ of $\GL_{n}$ for any $n\leq l-1$. This concludes the proof of the corollary.
 \end{proof}

\section{Twists by $\GL_{l}$}\label{section Twists by GLl}

In this section, we consider the twists by $\GL_l.$ These twists are significantly different from the previous twists. The embedding of $\SO_{2l}$ into $\SO_{2l+1}$ plays a crucial role in these arguments. For reference, we explicate the embedding here. Recall we embed $\SO_{2l}$ into $\SO_{2l+1}$ via
$$
g:=\left(\begin{matrix}
A & B \\
C & D
\end{matrix}\right)
\mapsto
M^{-1}
\left(\begin{matrix}
A & & B \\
 & 1 & \\
C & & D  
\end{matrix}\right)
M
$$
where $A, B, C,$ and $D$ are $l\times l$ matrices and
$$
M=\mathrm{diag}(I_{l-1},\left(\begin{matrix}
    0 & 1 & 0 \\
    \frac{1}{2} & 0 & \frac{1}{2\gamma} \\
    \frac{1}{2} & 0 & \frac{-1}{2\gamma} 
    \end{matrix}\right), I_{l-1}).
$$
For brevity, we set 
$$
\tilde{M}:=\left(\begin{matrix}
    0 & 1 & 0 \\
    \frac{1}{2} & 0 & \frac{1}{2\gamma} \\
    \frac{1}{2} & 0 & \frac{-1}{2\gamma} 
    \end{matrix}\right).
$$
We can further decompose the matrices $A,B,C,D$ as follows:
\begin{equation*}
A=\left(\begin{matrix}
A_{1,1} & A_{1,2} \\
A_{2,1} & A_{2,2}
\end{matrix}\right),   B=\left(\begin{matrix}
B_{1,1} & B_{1,2} \\
B_{2,1} & B_{2,2}
\end{matrix}\right), C=\left(\begin{matrix}
C_{1,1} & C_{1,2} \\
C_{2,1} & C_{2,2}
\end{matrix}\right), D=\left(\begin{matrix}
D_{1,1} & D_{1,2} \\
D_{2,1} & D_{2,2}
\end{matrix}\right),
\end{equation*}
where $A_{1,1},B_{1,2},C_{2,1},D_{2,2}$ are $(l-1)\times(l-1)$ matrices, $A_{1,2},B_{1,1},C_{2,2},D_{2,1}$ are $(l-1)\times 1$ matrices, $A_{2,1},B_{2,2},C_{1,1},D_{1,2}$ are $1\times(l-1)$ matrices, and $A_{2,2},B_{2,1},C_{1,2},D_{1,1}$ are $1\times 1$ matrices. Then the embedding of
$g$ in $\SO_{2l+1}$ is given by
\begin{equation}\label{embedding}
    \left(\begin{matrix}
    A_{1,1} & \left(\begin{matrix}
      A_{1,2} & 0 & B_{1,1}  
    \end{matrix}\right)\tilde{M} & B_{1,2} \\
    \tilde{M}\inv\left(\begin{matrix}
      A_{2,1} \\
      0 \\
      C_{1,1}
    \end{matrix}\right) & \tilde{M}\inv\left(\begin{matrix}
        A_{2,2} & 0 & B_{2,1} \\
        0 & 1 & 0 \\
        C_{1,2} & 0 & D_{1,1}
    \end{matrix}\right)\tilde{M} & \tilde{M}\inv\left(\begin{matrix}
        B_{2,2} \\
        0 \\
        D_{1,2}
    \end{matrix}\right) \\
    C_{2,1} & \left(\begin{matrix}
      C_{2,2} & 0 & D_{2,1}  
    \end{matrix}\right)\tilde{M} & D_{2,2} 
    \end{matrix}\right)
\end{equation}
Furthermore, we have
\begin{equation}\label{embedding details}
\begin{split}
    \left(\begin{matrix}
      A_{1,2} & 0 & B_{1,1}  \end{matrix}\right)\tilde{M}&=\left(\begin{matrix}
      \frac{B_{1,1}}{2} & A_{1,2} & \frac{-B_{1,1}}{2\gamma}  
    \end{matrix}\right), \\ \left(\begin{matrix}
      C_{2,2} & 0 & D_{2,1}  \end{matrix}\right)\tilde{M}&=\left(\begin{matrix}
      \frac{D_{2,1}}{2} & C_{2,2} & \frac{-D_{2,1}}{2\gamma}  
    \end{matrix}\right), \\
    \tilde{M}\inv\left(\begin{matrix}
      A_{2,1} \\
      0 \\
      C_{1,1}
    \end{matrix}\right)&=\left(\begin{matrix}
      C_{1,1} \\
      A_{2,1} \\
      -\gamma C_{1,1}
    \end{matrix}\right), \\ \tilde{M}\inv\left(\begin{matrix}
      B_{2,2} \\
      0 \\
      D_{1,2}
    \end{matrix}\right)&=\left(\begin{matrix}
      D_{1,2} \\
      B_{2,2} \\
      -\gamma D_{1,2}
    \end{matrix}\right), \\
    \tilde{M}\inv\left(\begin{matrix}
        A_{2,2} & 0 & B_{2,1} \\
        0 & 1 & 0 \\
        C_{1,2} & 0 & D_{1,1}
\end{matrix}\right)\tilde{M}&=\left(\begin{matrix}
        \frac{D_{1,1}+1}{2} & C_{1,2} & \frac{-D_{1,1}+1}{2\gamma} \\
        \frac{B_{2,1}}{2} & A_{2,2} & \frac{-B_{2,1}}{2\gamma} \\
        \frac{\gamma(1-D_{1,1})}{2} & -\gamma C_{1,2} & \frac{D_{1,1}+1}{2}
    \end{matrix}\right).
    \end{split}
\end{equation}

Let $\tau$ be an irreducible $\psi^{-1}$-generic representation of $\GL_l$, $v\in\tau$ be a fixed vector, and define $\xi_v\in I(\tau)$ by supp$(\xi_v)=L_{l} V_{l}= Q_{l}$ and 
$$
\xi_v(l_{l}(a)u)=\tau(a)v,$$ for any $a\in\GL_{l}, u\in V_{l}.
$
Let $f_v=f_{\xi_v}\in I(\tau,\psi^{-1}).$ That is $f_v(g,a)=\Lambda_\tau(\tau(a)\xi_v(g))$, where $\Lambda_\tau\in \mathrm{Hom}_{U_{GL_{l}}}(\tau,\psi^{-1}), a\in\GL_{l},$ and  $g\in\SO_{2l+1}$. Let $W_v(a)=\Lambda_\tau(\tau(a)v).$ The following proposition shows that the zeta integrals are again nonzero.

\begin{prop}\label{GL_lnonzero}
Suppose $\B_{\pi,\psi}$ be the normalized Bessel function of $\pi$, an irreducible cuspidal $\psi$-generic representation of $\SO_{2l}$. Then, $\Psi(\B_{\pi,\psi},f_v)= W_v(w_{l,l})$. Moreover, we may choose $v\in\tau$ such that $\Psi(\B_{\pi,\psi},f_v)\neq 0.$
\end{prop}

\begin{proof}
By definition.
$$
\Psi(\B_{\pi,\psi},f_v)=\sum_{g\in U_{\SO_{2l}}\setminus\SO_{2l}} \B_{\pi,\psi}(g)f_v(w_{l,l}g, I_{l}).
$$
By definition, the support of $f_v(\cdot, I_l)$ is $Q_l.$ It is straightforward to check that the preimage of $Q_l$ under the embedding of $\SO_{2l}$ into $\SO_{2l+1}$ is contained in the standard parabolic subgroup of $\SO_{2l}$ with Levi subgroup isomorphic to $\GL_{l-1}\times\SO_2$ (see Equations \ref{embedding} and \ref{embedding details}). Moreover, the preimage of $L_l$ consists of elements of the form $t_{l-1}(a)$ where $a\in\GL_{l-1}.$

Thus, $$
\Psi(\B_{\pi,\psi},f_v)=\sum_{a\in U_{\GL_{l-1}}\setminus\GL_{l-1}} \B_{\pi,\psi}(t_{l-1}(a))f_v(w_{l,l}l_{l-1}(a), I_{l}).
$$
By Proposition \ref{uppertriangular}, $\B_{\pi,\psi}(t_{l-1}(a))=0$ for any $a$ which isn't upper triangular. Thus, by Lemma \ref{center}, $\Psi(\B_{\pi,\psi},f_v)=W_v(w_{l,l}).$ If we let $\tilde{v}$ be a Whittaker vector for $\tau$ and $v=\tau(w_{l,l}\inv)\tilde{v},$ then $\Psi(\B_{\pi,\psi},f_v)=W_v(w_{l,l})=W_{\tilde{v}} (I_l)\neq 0.$ This proves the proposition.
 \end{proof}

The support of $\tilde{f}_v=M(\tau,\psi\inv)f_v$ lies in $Q_l w_l V_l.$ In order to compute the zeta integral, we need to find the preimage of this set inside the support of the Bessel functions. This is essentially the statement of the proposition below.

\begin{prop}\label{l embed}
Let $w\in\mathrm{B}(\SO_{2l})$ and $$t=\mathrm{diag}(t_1,\dots,t_{l-1},\left(\begin{matrix}
a & b\rho \\
b & a
\end{matrix}\right),t_{l-1}\inv,\dots,t_1\inv) \in T_{\SO_{2l}}.$$ Then $tw\in Q_l w_l V_l$ (via the embedding of $\SO_{2l}$ into $\SO_{2l+1}$) if and only if we have $w\in\mathrm{B}_{l-1}(\SO_{2l})$ and $a\neq 1$.
\end{prop}

\begin{proof}
We begin with a sketch of the proof. 
We first compute the embedding of the torus in coordinates and the embedding of the Weyl element $w$, then compute $Q_l w_l V_l$ in terms of block matrices. We see that the block matrix in the lower left corner must be invertible if $tw\in Q_l w_l V_l$ which gives us the condition we need. We proceed with the details of the proof.

Recall that the embedding sends $t$ to
$$
\mathrm{diag}(s,\left(\begin{matrix}
\frac{1}{2}(1+a) & b & \frac{1}{2\gamma}(1-a) \\
\gamma b & a & -b \\
\frac{\gamma}{2}(1-a) & -\gamma b & \frac{1}{2}(1+a)
\end{matrix}\right), s^*)\in\SO_{2l+1}
$$
where $s=\mathrm{diag}(t_1, t_2, \dots, t_{l-1})$. 
By Corollary \ref{Besselpartnonsplit}, there exists an integer $n$ with $1\leq n\leq l-1$ and $w'\in W(\GL_n)$ such that $w=t_n(w')\tilde{w}_n.$ The image of $t_n(w')$ in $\SO_{2l+1}$ is $\mathrm{diag}(w',I_{l-n-1},I_3,I_{l-n-1},(w')^*).$ The image of $\tilde{w}_n$ is 
$$
\left(\begin{matrix}
 & & & & I_n \\
 & I_{l-n-1} & & & \\
 & & \left(\begin{matrix}
     1 & & \\
     & -1 & \\
     & & 1
 \end{matrix}\right)^n & & \\
 & & & I_{l-n-1} & \\
 I_n & & & &
\end{matrix}\right).
$$
Altogether, we compute directly that the image of $tw$ in $\SO_{2l+1}$ is
\begin{equation}\label{eqn embed tw}
    \left(\begin{matrix}
 & & & & s_nw' \\
 & s_{l-1} & & & \\
 & & \left(\begin{matrix}
\frac{1}{2}(1+a) & b & \frac{1}{2\gamma}(1-a) \\
\gamma b & a & -b \\
\frac{\gamma}{2}(1-a) & -\gamma b & \frac{1}{2}(1+a)
\end{matrix}\right)\left(\begin{matrix}
     1 & & \\
     & -1 & \\
     & & 1
 \end{matrix}\right)^n & & \\
 & & & s_{l-1}^* & \\
 s_n^*(w')^* & & & &
\end{matrix}\right)
\end{equation}
where $s_n=\mathrm{diag}(t_1,\dots,t_n), s_{l-1}=\mathrm{diag}(t_{n+1},\dots,t_{l-1}), s_{l-1}^*=\mathrm{diag}(t_{l-1}\inv,\dots,t_{n+1}\inv),$ and $s_n^*=\mathrm{diag}(t_n\inv,\dots,t_1\inv).$

Next, we turn towards computing $Q_l w_l V_l.$ Let $a\in \GL_l$ and $n_1,n_2\in V_l$. We wish to compute $l_l(a)n_1 w_l n_2.$ Write 
$$
n_1=\left(\begin{matrix}
I_l & X & Y \\
    & 1 & X^\prime \\
    &   & I_l
\end{matrix}\right),
$$
where $X\in\mathbb{F}_q^l$, $Y\in\mathrm{Mat}_{l\times l}(\mathbb{F}_q),$ $X^\prime=-{}^t X J_l$, and $J_l Y + {}^t X^\prime\cdot X^\prime + {}^t Y J_l =0.$
Similarly, we write$$
n_2=\left(\begin{matrix}
I_l & M & N \\
    & 1 & M^\prime \\
    &   & I_l
\end{matrix}\right).
$$
Then, 
$$l_l(a)n_1 w_l n_2=
\left(\begin{matrix}
aY & aX+aYM & a+aX M^\prime+aYN \\
X^\prime    & 1+X^\prime M & M^\prime+X^\prime N \\
a^*    & a^*M  & a^*N
\end{matrix}\right).
$$

Thus, $tw=l_l(a)n_1 w_l n_2$ has a solution as long as we can find a suitable $a^*.$ From Equation \ref{eqn embed tw}, we see that this is only possible if $n=l-1.$ In this case, we require $a^*=\left(\begin{matrix}
& \frac{\gamma}{2}(1-a) \\
\mathrm{diag}(t_{l-1}\inv,\dots,t_{1}\inv)(w')^* &
\end{matrix}\right)$. Such $a^*$ exists as long as this matrix is invertible. Hence we must have $a\neq 1.$ This proves the proposition.
 \end{proof}

The following theorem shows that the twists by $\GL_l$ determine the Bessel functions, only up to conjugation however, on the part of the Bruhat cells that were missing in Theorem \ref{GL_{l-1}}.

\begin{thm}\label{GL_l}
Let $\pi$ and $\pi^\prime$ be irreducible cuspidal $\psi$-generic representations of $\SO_{2l}$ with the same central character. If $\gamma(\pi\times\tau,\psi)=\gamma(\pi^\prime\times\tau,\psi)$ for all irreducible generic representations $\tau$ of $\GL_l,$
then 
 $$(\B_{\pi,\psi}+\B_{\pi^c,\psi})(tw)=(\B_{\pi^\prime,\psi}+\B_{\pi'{}^c,\psi})(tw),$$ for any $t=\mathrm{diag}(t_1,\dots,t_{l-1},\left(\begin{matrix}
a & b\rho \\
b & a
\end{matrix}\right),t_{l-1}\inv,\dots,t_1\inv)\in T_{\SO_{2l}}$ with $a\neq 1$ and $w\in \mathrm{B}_{l-1}(\SO_{2l}).$  
\end{thm}

\begin{proof}
By Proposition \ref{GL_lnonzero}, we have $\Psi(\B_{\pi,\psi},f_v)=\Psi(\B_{\pi',\psi},f_v)=W_v(w_{l,l})$ and we can choose this to be nonzero. By assumption, $\gamma(\pi\times\tau,\psi)=\gamma(\pi^\prime\times\tau,\psi)$ for all irreducible generic representations $\tau$ of $\GL_l$, and so $\Psi(\B_{\pi,\psi},\tilde{f}_v)=\Psi(\B_{\pi',\psi},\tilde{f}_v).$
By definition,
$$
\Psi(\B_{\pi,\psi},\tilde{f}_v)=\sum_{g\in U_{\SO_{2l}}\setminus\SO_{2l}} \B_{\pi,\psi}(g)\tilde{f}_v(w_{l,l}g, I_{l}),
$$ and hence
$$
0=\sum_{g\in U_{\SO_{2l}}\setminus\SO_{2l}}( \B_{\pi,\psi}-\B_{\pi',\psi})(g)\tilde{f}_v(w_{l,l}g, I_{l}).
$$
Let $T_a$ denote the set of $t=\mathrm{diag}(t_1,\dots,t_{l-1},\left(\begin{matrix}
a & b\rho \\
b & a
\end{matrix}\right),t_{l-1}\inv,\dots,t_1\inv)\in T_{\SO_{2l}}$ such that $a\neq \pm1.$ 
The support of $\tilde{f}_v=M(\tau,\psi\inv)f_v$ lies in $Q_l w_l V_l.$ From the Bruhat decomposition and Propositions \ref{Besselprop} and \ref{l embed} and equality of the central characters (to handle the case that $a=-1$), we have
$$
0=\sum_{t\in T_a, w\in \mathrm{B}_{l-1}(\SO_{2l})}( \B_{\pi,\psi}-\B_{\pi',\psi})(tw)\tilde{f}_v(w_{l,l}tw, I_{l}).
$$

 For $ w\in \mathrm{B}_{l-1}(\SO_{2l})$, there exists $w'\in W(\GL_{l-1})$ such that $w=t_{l-1}(w')\tilde{w}_{l-1}.$ Then, for $t\in T_a$, it follows from the proof of Proposition \ref{l embed} that $w_{l,l}tw=l_l(\gamma A)w_lx$ where $A\in\GL_l$ and $x\in V_l$ with
 $$A=\left(\begin{matrix}
& \frac{\gamma}{2}(1-a) \\
\mathrm{diag}(t_{l-1}\inv,\dots,t_{1}\inv)(w')^* &
\end{matrix}\right)^*.$$ By Lemma \ref{intertwinenon}, $\tilde{f}_v(w_{l,l}tw, I_{l})=W^*_v(\gamma A).$

Note that the assignment $tw\mapsto A$ is a two-to-one map. Indeed, recall that for $$t=\mathrm{diag}(t_1,\dots,t_{l-1},\left(\begin{matrix}
a & b\rho \\
b & a
\end{matrix}\right),t_{l-1}\inv,\dots,t_1\inv)\in T_{\SO_{2l}},$$
we have
$$ctc=\mathrm{diag}(t_1,\dots,t_{l-1},\left(\begin{matrix}
a & -b\rho \\
-b & a
\end{matrix}\right),t_{l-1}\inv,\dots,t_1\inv)\in T_{\SO_{2l}}.$$
Since $t\in T_a,$ it follows that $b\neq 0$ and hence $t\neq ctc$ (in fact, $T_a$ is precisely the subset of $T_{\SO_{2l}}$ which is not fixed by conjugation by $c$). However, $ctc\in T_a$ both $tw$ and $ctcw$ result in the same $A.$ As a result, we cannot apply Lemma \ref{Niennon} directly yet. 

To remedy this, we partition $T_a$ into two sets $T_1$ and $T_2$ with the property that $t\in T_1$ if and only if $ctc\in T_2.$ Then
 \begin{align*}
     0=&\sum_{t\in T_a, w\in \mathrm{B}_{l-1}(\SO_{2l})}( \B_{\pi,\psi}-\B_{\pi',\psi})(tw)\tilde{f}_v(w_{l,l}tw, I_{l}) \\
     =&\sum_{t\in T_1, w\in \mathrm{B}_{l-1}(\SO_{2l})}( \B_{\pi,\psi}-\B_{\pi',\psi})(tw)\tilde{f}_v(w_{l,l}tw, I_{l}) \\
     +&\sum_{t\in T_2, w\in \mathrm{B}_{l-1}(\SO_{2l})}( \B_{\pi,\psi}-\B_{\pi',\psi})(tw)\tilde{f}_v(w_{l,l}tw, I_{l}) \\
     =&\sum_{t\in T_1, w\in \mathrm{B}_{l-1}(\SO_{2l})}( \B_{\pi,\psi}-\B_{\pi',\psi})(tw)\tilde{f}_v(w_{l,l}tw, I_{l}) \\
     +&\sum_{t\in T_1, w\in \mathrm{B}_{l-1}(\SO_{2l})}( \B_{\pi,\psi}-\B_{\pi',\psi})(ctcw)\tilde{f}_v(w_{l,l}ctcw, I_{l}).
 \end{align*}
 Recall we set $\tilde{t}=\mathrm{diag}(I_{l-1},-1,-1,I_{l-1}).$ Conjugation by $\tilde{t}$ fixes $T_1.$ Hence, by Proposition
 \ref{Prop Besselconj}, we have
 \begin{align*}
     &\sum_{t\in T_1, w\in \mathrm{B}_{l-1}(\SO_{2l})}( \B_{\pi,\psi}-\B_{\pi',\psi})(ctcw)\tilde{f}_v(w_{l,l}ctcw, I_{l}) \\
     =&\sum_{t\in T_1, w\in \mathrm{B}_{l-1}(\SO_{2l})}( \B_{\pi,\psi}-\B_{\pi',\psi})(\tilde{t}\inv ctcw\tilde{t})\tilde{f}_v(w_{l,l}ctcw, I_{l}) \\
     =&\sum_{t\in T_1, w\in \mathrm{B}_{l-1}(\SO_{2l})}( \B_{\pi,\psi}-\B_{\pi',\psi})(c\tilde{t}\inv t\tilde{t}cw)\tilde{f}_v(w_{l,l}ctcw, I_{l} \\
     =&\sum_{t\in T_1, w\in \mathrm{B}_{l-1}(\SO_{2l})}( \B_{\pi^c,\psi}-\B_{\pi'{}^c,\psi})(tw)\tilde{f}_v(w_{l,l}ctcw, I_{l}).
 \end{align*}
 Thus, we have
 \begin{align*}
     0=&\sum_{t\in T_1, w\in \mathrm{B}_{l-1}(\SO_{2l})}( \B_{\pi,\psi}-\B_{\pi',\psi})(tw)\tilde{f}_v(w_{l,l}tw, I_{l}) \\
     +&\sum_{t\in T_1, w\in \mathrm{B}_{l-1}(\SO_{2l})}( \B_{\pi,\psi}-\B_{\pi',\psi})(ctcw)\tilde{f}_v(w_{l,l}ctcw, I_{l}) \\
     =&\sum_{t\in T_1, w\in \mathrm{B}_{l-1}(\SO_{2l})}( \B_{\pi,\psi}-\B_{\pi',\psi})(tw)\tilde{f}_v(w_{l,l}tw, I_{l}) \\
     +&\sum_{t\in T_1, w\in \mathrm{B}_{l-1}(\SO_{2l})}( \B_{\pi^c,\psi}-\B_{\pi'{}^c,\psi})(tw)\tilde{f}_v(w_{l,l}ctcw, I_{l}).
 \end{align*}
 Note that $\tilde{f}_v(w_{l,l}tw, I_{l})=\tilde{f}_v(w_{l,l}ctcw, I_{l})=W^*_v(\gamma A).$ Thus,
  \begin{align*}
     0=&\sum_{t\in T_1, w\in \mathrm{B}_{l-1}(\SO_{2l})}( \B_{\pi,\psi}-\B_{\pi',\psi})(tw)\tilde{f}_v(w_{l,l}tw, I_{l}) \\
     +&\sum_{t\in T_1, w\in \mathrm{B}_{l-1}(\SO_{2l})}( \B_{\pi^c,\psi}-\B_{\pi'{}^c,\psi})(tw)\tilde{f}_v(w_{l,l}ctcw, I_{l}) \\
     =&\sum_{t\in T_1, w\in \mathrm{B}_{l-1}(\SO_{2l})}( \B_{\pi,\psi}-\B_{\pi',\psi}+\B_{\pi^c,\psi}-\B_{\pi'{}^c,\psi})(tw)\tilde{f}_v(w_{l,l}tw, I_{l}).
 \end{align*}
 Recall that
 \begin{align*}
     A&=\left(\begin{matrix}
& \frac{\gamma}{2}(1-a) \\
\mathrm{diag}(t_{l-1}\inv,\dots,t_{1}\inv)(w')^* &
\end{matrix}\right)^* \\
&=\left(\mathrm{diag}(\frac{\gamma}{2}(1-a),t_{l-1}\inv,\dots,t_{1}\inv)\left(\begin{matrix}
    &  1 \\
    (w')^* &
\end{matrix}\right)\right)^*.
 \end{align*}
Let $X'$ be the subset of $\GL_l$ defined by 
the image of the map $tw\mapsto A$ and set $X:=\gamma X'.$ The map $tw\mapsto \gamma A$ gives a bijection from $T_1$ to $X,$ which denote by $\xi.$ We define a function $f:X\rightarrow\mathbb{C}$ by $f(x)=(\B_{\pi,\psi}-\B_{\pi',\psi}+\B_{\pi^c,\psi}-\B_{\pi'{}^c,\psi})(\xi\inv( x)).$

Then
\begin{align*}
     0=&\sum_{t\in T_1, w\in \mathrm{B}_{l-1}(\SO_{2l})}( \B_{\pi,\psi}-\B_{\pi',\psi}+\B_{\pi^c,\psi}-\B_{\pi'{}^c,\psi})(tw)\tilde{f}_v(w_{l,l}tw, I_{l}) \\
     = &\sum_{x\in X} f(x)W_v^*( x).
 \end{align*}
Finally, from Lemma \ref{Niennon}, we have $f(x)=0$ for any $x\in X.$ Therefore,
$$
(\B_{\pi,\psi}-\B_{\pi',\psi}+\B_{\pi^c,\psi}-\B_{\pi'{}^c,\psi})(tw)=0,
$$
for any $t\in T_1$ and $w\in \mathrm{B}_{l-1}(\SO_{2l}).$ Conjugation by $c$ fixes $\mathrm{B}_{l-1}(\SO_{2l})$ and sends $T_1$ to $T_2$. Also $$(\B_{\pi,\psi}-\B_{\pi',\psi}+\B_{\pi^c,\psi}-\B_{\pi'{}^c,\psi})(tw)=(\B_{\pi,\psi}-\B_{\pi',\psi}+\B_{\pi^c,\psi}-\B_{\pi'{}^c,\psi})(ctcw).$$
Thus we obtain
$$(\B_{\pi,\psi}-\B_{\pi',\psi}+\B_{\pi^c,\psi}-\B_{\pi'{}^c,\psi})(tw)=0,
$$
for any $t\in T_2$ and $w\in \mathrm{B}_{l-1}(\SO_{2l}).$ Since $T_a=T_1\cup T_2,$ we have proven the theorem.
\end{proof}

The following corollary, when combined with Corollary \ref{conj n gamma}, shows that $\gamma$-factor is unable to distinguish between a representation and its conjugate. That is, $\gamma(\pi\times\tau,\psi)=\gamma(\pi^c\times\tau,\psi)$ for all irreducible generic representations $\tau$ of $\GL_{n}$ with $n\leq l$.

\begin{cor}\label{conj l gamma}
Let $\pi$ be an irreducible cuspidal $\psi$-generic representation of $\SO_{2l}$. Then we have $\gamma(\pi\times\tau,\psi)=\gamma(\pi^c\times\tau,\psi)$ for all irreducible generic representations $\tau$ of $\GL_{l}$.
\end{cor}

\begin{proof}
By Corollary \ref{conj n gamma}, $\gamma(\pi\times\tau,\psi)=\gamma(\pi^c\times\tau,\psi)$ for all irreducible generic representations $\tau$ of $\GL_{n}$ with $n\leq l-1$. So, by Theorems \ref{GL_n} and \ref{GL_{l-1}}, 
\begin{align*}
&\quad\Psi(\B_{\pi,\psi},\tilde{f}_v)-\Psi(\B_{\pi^c,\psi},\tilde{f}_v) \\ &=\sum_{t\in T_a, w\in \mathrm{B}_{l-1}(\SO_{2l}), u\in U_{\SO_{2l}}}
(\B_{\pi,\psi}-\B_{\pi^c,\psi})(twu)\tilde{f}_v(w_{l,l}twu, I_l).
\end{align*}
Conjugation by $c$ map $T_a,$ $B_{l-1}(\SO_{2l}),$ and $U_{\SO_{2l}}$ to themselves and sends $\B_{\pi,\psi}$ to $\B_{\pi,\psi}$ by Proposition \ref{Prop Besselconj}. It also fixes $\tilde{f}_v(w_{l,l}twu, I_l)$. Hence 
\begin{align*}
&\sum_{t\in T_a, w\in \mathrm{B}_{l-1}(\SO_{2l}), u\in U_{\SO_{2l}}}
(\B_{\pi,\psi}-\B_{\pi^c,\psi})(twu)\tilde{f}_v(w_{l,l}twu, I_l) \\
=&\sum_{t\in T_a, w\in \mathrm{B}_{l-1}(\SO_{2l}), u\in U_{\SO_{2l}}}
(\B_{\pi^c,\psi}-\B_{\pi^c,\psi})(twu)\tilde{f}_v(w_{l,l}twu, I_l) \\
=&0.
\end{align*}
By Proposition \ref{GL_lnonzero}, we may choose a nonzero $v\in\tau$ such that $\Psi(\B_{\pi,\psi},f_v)=\Psi(\B_{\pi^c,\psi},f_v)=W_v(w_{l,l})\neq 0.$ Thus, we have $\gamma(\pi\times\tau,\psi)=\gamma(\pi^c\times\tau,\psi)$. This proves the corollary.
 \end{proof}

\section{The converse theorem}\label{section The converse theorem}

In this section, we prove the converse theorem. First, we combine the results of the previous sections to obtain the following theorem.

\begin{thm}\label{Bessels Equal}
Let $\pi$ and $\pi^\prime$ be irreducible cuspidal $\psi$-generic representations of  $\SO_{2l}$ with the same central character. If $$\gamma(\pi\times\tau,\psi)=\gamma(\pi^\prime\times\tau,\psi),$$
for all irreducible generic representations $\tau$ of $\GL_n$ with $n\leq l,$
then we have that $$(\B_{\pi,\psi}+\B_{\pi^c,\psi})(g)=(\B_{\pi^\prime,\psi}+\B_{\pi'{}^c,\psi})(g)$$  for any $g\in\SO_{2l}(\mathbb{F}_q)$.
\end{thm}

\begin{proof}
By the Bruhat decomposition, we may assume $g=u_1 t w u_2\in B_{\SO_{2l}}wB_{\SO_{2l}}$ for some $w\in W(\SO_{2l}).$ By Proposition \ref{Besselprop} and the definition of the Bessel support, it is enough to show that 
\begin{equation}\label{eqn Bessel conj equal}
    (\B_{\pi,\psi}+\B_{\pi^c,\psi})(tw)=(\B_{\pi^\prime,\psi}+\B_{\pi'{}^c,\psi})(tw)
\end{equation}
for any $t\in T_{\SO_{2l}}$ and $w\in \mathrm{B}(\SO_{2l}).$ By Proposition \ref{Besselpartnonsplit}, we only need to show Equation \eqref{eqn Bessel conj equal} for $w \in \mathrm{B}_n(\SO_{2l})$ for $n\leq l-1.$
Since $\mathrm{B}_0(\SO_{2l})=\{I_{2l}\},$ we have $(\B_{\pi,\psi}+\B_{\pi^c,\psi})(t)=(\B_{\pi^\prime,\psi}+\B_{\pi'{}^c,\psi})(t)$ for any $t\in T_{\SO_{2l}}$ by Lemma \ref{center} and equality of the central characters. Equation \eqref{eqn Bessel conj equal} for $1\leq n\leq l-2$ follows from the equality of the central characters, Proposition \ref{prop support on B_n}, Theorem \ref{GL_n}, and Corollary \ref{cGL_n}. Equation \eqref{eqn Bessel conj equal} for $n=l-1$ follows from the equality of the central characters, Theorems \ref{GL_{l-1}} and \ref{GL_l}, and Corollary \ref{cGL_n}.
 \end{proof}

 We now prove the converse theorem. The proof is essentially the same as in the split case since we have established Theorem \ref{Bessels Equal}. However, we include it for completeness.

\begin{thm}[The Converse Theorem for $\SO_{2l}$]\label{converse thm}
Let $\pi$ and $\pi^\prime$ be irreducible cuspidal $\psi$-generic representations of  $\SO_{2l}$ with the same central character. If $$\gamma(\pi\times\tau,\psi)=\gamma(\pi^\prime\times\tau,\psi),$$ 
for all irreducible generic representations $\tau$ of $\GL_n(\mathbb{F}_q)$ with $n\leq l,$
then $\pi\cong\pi'$ or $\pi\cong \pi'{}^c.$
\end{thm}

\begin{proof}
By Theorem \ref{Bessels Equal}, $(\B_{\pi,\psi}+\B_{\pi^c,\psi})=(\B_{\pi^\prime,\psi}+\B_{\pi'{}^c,\psi})$ on all of $\SO_{2l}.$ We let $W_\pi, W_{\pi^c}, W_{\pi^\prime},$ and $W_{\pi'{}^c}$ be the Whittaker models of $\pi, \pi^c, \pi^\prime$, and $\pi'{}^c$ respectively.

First, suppose that $\pi$ is isomorphic to $\pi^c$. By uniqueness of Whittaker models, we have $\B_{\pi,\psi}+\B_{\pi^c,\psi}=2\B_{\pi,\psi} \in W_\pi$. Then, $2\B_{\pi,\psi}=\B_{\pi^\prime,\psi}+\B_{\pi'{}^c,\psi}$ and hence $W_\pi \cap (W_{\pi^\prime} \oplus W_{\pi'{}^c})\neq 0.$ Since $W_\pi$ is isomorphic to $\pi$ and is hence irreducible, we have that $W_\pi \cap (W_{\pi^\prime} \oplus W_{\pi'{}^c})=W_\pi$ and therefore, by uniqueness of Whittaker models, $W_\pi$ must be isomorphic to one of  $W_{\pi^\prime}$ or $W_{\pi'{}^c}$.
 
Second, suppose $\pi$ is not isomorphic to $\pi^c$. We have that $\B_{\pi,\psi}\in W_\pi$. We also have that $\B_{\pi^\prime,\psi}+\B_{\pi'{}^c,\psi}-\B_{\pi^c,\psi}\in W_{\pi^\prime} \oplus W_{\pi'{}^c}\oplus W_{\pi^c}$. Thus, $W_\pi \cap (W_{\pi^\prime} \oplus W_{\pi'{}^c}\oplus W_{\pi^c})\neq 0$ and hence the intersection is a nonzero subrepresentation of $W_\pi$. Since $W_\pi$ is isomorphic to $\pi$ and is hence irreducible, we have that $W_\pi \cap (W_{\pi^\prime} \oplus W_{\pi'{}^c}\oplus W_{\pi^c}) =W_\pi$ and therefore, by uniqueness of Whittaker models, $W_\pi$ must be isomorphic to one of  $W_{\pi^c}, W_{\pi^\prime},$ or $W_{\pi'{}^c}$. By assumption, $W_\pi$ is not isomorphic to $W_{\pi^c}$ and hence must be isomorphic to $W_{\pi^\prime}$ or $W_{\pi'{}^c}$. 

Therefore, in either case, we have shown the converse theorem.
 \end{proof}

\bibliographystyle{amsplain}
\bibliography{quasisplit_finite_even_orthog}

\end{document}